\documentclass[10pt, a4paper,reqno]{amsart}

\usepackage[utf8]{inputenc}
\usepackage[T1]{fontenc}
\usepackage{paralist,charter}
\usepackage{amssymb}
\usepackage{amsmath}
\usepackage{amsthm}
\usepackage[foot]{amsaddr}
\usepackage[ngerman,english]{babel}
\usepackage[noadjust]{cite}
\usepackage[hang]{caption}
\usepackage{fancyhdr}
\usepackage{marvosym}
\usepackage{chngcntr}
\usepackage{multicol}
\usepackage{mathtools}

\numberwithin{equation}{section}
\numberwithin{figure}{section}

\newtheoremstyle{mythm}{3pt}{3pt}{\itshape}{0pt}{\bfseries}{.}{0.5eM}{}
\theoremstyle{mythm}

\newtheorem{definition}{Definition}[section]
\newtheorem{thm}[definition]{Theorem}

\newtheorem{lem}[definition]{Lemma}
\newtheorem{cor}[definition]{Corollary}  
\newtheorem{prop}[definition]{Proposition}

\newtheoremstyle{myrem}{3pt}{3pt}{\normalfont}{0pt}{\bfseries}{.}{0.5em}{}
\theoremstyle{myrem}

\newtheorem{rem}[definition]{Remark}

\newcommand{\equi}{\ensuremath{\Leftrightarrow}}

\newcommand{\ld}{\ensuremath{,\ldots,}}
\newcommand{\ssq}{\ensuremath{\subseteq}}
\newcommand{\smin}{\ensuremath{\setminus}}
\newcommand{\eps}{\ensuremath{\epsilon}}

\newcommand{\htop}{\ensuremath{h_{\mathrm{top}}}}
\newcommand{\dist}{\ensuremath{\mathop{\rm dist}}}

\newcommand{\Leb}{\ensuremath{\mathrm{Leb}}}
\newcommand{\interior}{\ensuremath{\mathrm{int}}}
\newcommand{\inte}{\ensuremath{\mathrm{int}}}
\newcommand{\closure}{\ensuremath{\mathrm{cl}}}

\newcommand{\spann}{\ensuremath{\mathrm{span}}}

\newcommand{\kreis}{\ensuremath{\mathbb{T}^{1}}}

\newcommand{\nfolge}[1]{\ensuremath{(#1)_{n\in\mathbb{N}}}}

\newcommand{\alphlist}{\begin{list}{(\alph{enumi})}{\usecounter{enumi}\setlength{\parsep}{2pt}
      \setlength{\itemsep}{1pt} \setlength{\topsep}{5pt}
      \setlength{\partopsep}{3pt}}}
\newcommand{\arablist}{\begin{list}{(\arabic{enumi})}{\usecounter{enumi}\setlength{\parsep}{2pt}
          \setlength{\itemsep}{1pt} \setlength{\topsep}{5pt}
          \setlength{\partopsep}{3pt}}}
\newcommand{\romanlist}{\begin{list}{(\roman{enumi})}{\usecounter{enumi}\setlength{\parsep}{2pt}
              \setlength{\itemsep}{1pt} \setlength{\topsep}{5pt}
              \setlength{\partopsep}{3pt}}}
\newcommand{\Romanlist}{\begin{list}{(\Roman{enumi})}{\usecounter{enumi}\setlength{\parsep}{2pt}
              \setlength{\itemsep}{1pt} \setlength{\topsep}{5pt}
              \setlength{\partopsep}{3pt}}}
\newcommand{\bulletlist}{\begin{list}{$\bullet$}{\setlength{\parsep}{2pt}
                \setlength{\itemsep}{1pt} \setlength{\topsep}{5pt}
                \setlength{\partopsep}{3pt}\setlength{\leftmargin}{15pt}}} 
\newcommand{\Alphlist}{\begin{list}{(\Alph{enumi})}{\usecounter{enumi}\setlength{\parsep}{2pt}
      \setlength{\itemsep}{1pt} \setlength{\topsep}{5pt}
      \setlength{\partopsep}{3pt}}}
 \newcommand{\listend}{\end{list}}

\newcommand{\T}{\ensuremath{\mathbb{T}}}

\newcommand{\N}{\ensuremath{\mathbb{N}}} 
\newcommand{\R}{\ensuremath{\mathbb{R}}}
\newcommand{\Z}{\ensuremath{\mathbb{Z}}}
\newcommand{\Q}{\ensuremath{\mathbb{Q}}}

\newcommand{\cC}{\mathcal{C}}
\newcommand{\cD}{\mathcal{D}}

\newcommand{\cF}{\mathcal{F}}
\newcommand{\cG}{\mathcal{G}}

\newcommand{\cK}{\mathcal{K}}
\newcommand{\cL}{\mathcal{L}}

\newcommand{\cP}{\mathcal{P}}

\newcommand{\nLim}{\ensuremath{\lim_{n\rightarrow\infty}}}

\newcommand{\jLim}{\ensuremath{\lim_{j\rightarrow\infty}}}

\newcounter{paranum}[section]

\newcommand{\mc}{\mathcal}
\newcommand{\w}{\omega}
\newcommand{\f}{R_\w}
\renewcommand{\:}{\colon}
\renewcommand{\epsilon}{\varepsilon}
\newcommand{\rH}{\Theta_H^r}

\DeclareMathOperator{\GL}{GL}

\newcommand{\oplam}{\mbox{\Large $\curlywedge$}}

\title[Irregular model sets and tame dynamics]{Irregular model sets and tame
  dynamics} \author{G.~Fuhrmann$^1$} \address{$^1$ {\em Corresponding author.}
  Department of Mathematics, Imperial College London, 180 Queen’s Gate, London
  SW7 2AZ, UK.}  \author{E.~Glasner$^2$}\address{$^2$Department of Mathematics,
  Tel-Aviv University, Ramat Aviv, Israel.} \author{T.~J\"ager$^3$}
\author{C.~Oertel$^3$} \address{$^3$Institute of Mathematics, Friedrich Schiller
  University Jena, Germany.} \email{gabriel.fuhrmann@imperial.ac.uk,
  glasner@math.tau.ac.il} \email{tobias.jaeger@uni-jena.de,
  christian.oertel@uni-jena.de}

\begin{document}

 \maketitle

 \begin{abstract}
  We study the dynamical properties of irregular model sets and show that the
  translation action on their hull always admits an infinite independence
  set. The dynamics can therefore not be tame and the topological sequence
  entropy is strictly positive. Extending the proof to a more general setting,
  we further obtain that tame implies regular for almost automorphic group
  actions on compact spaces.

 In the converse direction, we show that even in the restrictive case of
 Euclidean cut and project schemes irregular model sets may be uniquely ergodic 
 and have zero topological entropy. This provides negative answers to questions by
 Schlottmann and Moody in the Euclidean setting.


 \end{abstract}

 \section{Introduction}

In the mathematical theory of quasicrystals and aperiodic order, one of the
major constructions of aperiodic structures is the {\em cut and project method},
introduced by Meyer in the context of algebraic number theory
\cite{Meyer1972AlgebraicNumbers}. The aim of this work is to contribute to a
better understanding of the relations between the different ingredients in this
construction and the dynamical properties of the resulting Delone dynamical
systems. More precisely, we study irregular model sets which are obtained when
the compact window in the cut and project construction has a positive measure
boundary. In contrast to regular model sets, whose dynamics and diffraction
theory are rather well-understood
\cite{Schlottmann1999GeneralizedModelSets,LeeMoodySolomyak2002,BaakeLenz2004PurePointDiffractionSpectrum,BaakeLenzMoody2007Characterization},
the description of their irregular counterparts is still far from being
satisfactory. As a byproduct, it turns out that the cut and project method also
provides an alternative approach to problems in symbolic and topological
dynamics outside the classical focus of aperiodic order. Our main results can be
stated as follows. We refer to Section~\ref{Preliminaries} for definitions and
background.

\begin{thm} \label{t.tame_implies_regular_for_model_sets}
  Suppose that $\oplam(W)$ is an irregular model set, arising from a cut and
  project scheme $(G,H,\cL)$ with locally compact and second countable abelian
  groups $G$ and $H$ and co-compact lattice $\cL\ssq G\times H$. Then there
  exists an infinite independence set for the dynamical hull, and consequently
  the translation action on the hull is not tame.
\end{thm}

We note that the above conclusions also hold for regular model sets whose internal group
is the circle and whose window has a Cantor set boundary, see
Theorem~\ref{t.irregular_implies_free_set}. The question whether tame implies
regular has actually been asked first in the more general context of topological
group actions \cite{Glasner2018MinimalTameSystems}. By modifying
the proof of Theorem~\ref{t.tame_implies_regular_for_model_sets}, the above
result can be extended to this setting.
 Hence, we obtain the following positive answer to
 \cite[Problem~5.7]{Glasner2018MinimalTameSystems}.

\begin{thm} \label{t.tame_implies_regular_for_group_actions}
  Suppose that $(X,T)$ is an almost automorphic topological group action.  If
  $(X,T)$ is tame, then it is a regular extension of its maximal equicontinuous
  factor.
\end{thm}

 Due to Theorem~\ref{t.tame_implies_regular_for_model_sets}, non-tameness can be
 seen as the minimal dynamical complexity an irregular model set must
 exhibit. As a direct consequence, one also obtains positive topological
 sequence entropy.  It is natural to ask if there are further or stronger
 dynamical implications of irregularity. In particular, Moody has raised the
 question whether irregular model sets need to have positive topological entropy
 (see \cite{PleasantsHuck2013KFreePoints}), and Schlottmann suggested that they
 cannot be uniquely ergodic \cite{Schlottmann1999GeneralizedModelSets}.

In the general setting of cut and project schemes, however, it is not too
difficult to give negative answers to these questions. The reason is that any
Toeplitz sequence can be interpreted as a model set
\cite{BaakeJaegerLenz2016ToeplitzFlowsModelSets}, and examples of uniquely
ergodic and zero entropy irregular Toeplitz flows have long been known
\cite{Williams1984ToeplitzFlows,MarkleyPaul1979}. The situation is different in
the more restrictive setting of Euclidean cut and project schemes, where both
questions were still completely open. Using methods from low-dimensional
dynamics, we construct counterexamples and obtain the following.

\begin{thm} \label{t.counterexamples}
  There exist irregular model sets arising from Euclidean cut and project
  schemes such that the translation action on the dynamical hull is uniquely
  ergodic and has zero topological entropy.
\end{thm}

In fact, we obtain two different types of examples: in the first case, the
translation action is an at most two-to-one extension of its maximal
equicontinuous factor and has zero entropy, but exhibits two distinct ergodic
invariant measures. In the second case, the translation action is mean equicontinuous and hence, in particular, 
uniquely ergodic with zero entropy.
In this case, the fibres of the
factor map onto the maximal equicontinuous factor are almost surely countably infinite. \medskip

The paper is organised as follows. The required preliminaries are provided in
Section~\ref{Preliminaries}. Theorems~\ref{t.tame_implies_regular_for_model_sets}
and \ref{t.tame_implies_regular_for_group_actions} are proven in separate
subsections of Section~\ref{TameImpliesRegular}. Readers who are mainly
interested in the result on minimal group actions may directly start with
Section~\ref{GroupActions}
and continue
afterwards with Proposition~\ref{p.window_translates} which plays a crucial
role in the proof of Theorem~\ref{t.tame_implies_regular_for_group_actions}.

The remaining sections are devoted to the construction and study of uniquely
ergodic and zero entropy examples in Euclidean CPS. General criteria for these
dynamical properties in terms of the window structure are provided in
Section~\ref{Self-similarity}, whereas the actual construction of windows with
the required properties is carried out in
Section~\ref{MainConstruction}. Finally, we discuss some implications of these
constructions for the (non-continuous) dependence of entropy on the window and
the diffraction spectra of our examples in Sections~\ref{MainConstruction} and \ref{Spectra}, respectively. \bigskip

\noindent{\bf Acknowledgements.} Parts of the above results have been obtained
during a workshop on {\em Spectral structures and topological methods in
  mathematical quasicrystals}, held in Oberwolfach, 1-7 October 2017. We thank
both the MFO and the organisers M.~Baake, D.~Damanik, J.~Kellendonk and D.~Lenz
for creating this opportunity. TJ has been supported by a Heisenberg
professorship of the German Research Council (DFG grant OE 538/6-1).  Moreover,
this project has received funding from the European Union's Horizon 2020
research and innovation programme under the Marie Sklodowska-Curie grant
agreement No 750865.

\section{Preliminaries} 
\label{Preliminaries}

 \subsection{Some topological dynamics}\label{TopDyn}

 Let $(X,T,\phi)$ be a {\em topological dynamical system}, that is, $T$ is a
 topological group, $X$ a Hausdorff topological space and $\phi$ a continuous
 left action of $T$ on $X$ by homeomorphisms on $X$.  We write $tx$ for the
 image $\phi(t,x)$ of the action of $t\in T$ on $x\in X$.  Most of the time, we
 keep the action $\phi$ implicit and simply refer to $(X,T)$ as a topological
 dynamical system.  In all of the following, $X$ is assumed to be compact
 metric.  $(X,T)$ is called \emph{minimal} if the orbit of every point $x\in X$
 is dense in $X$, that is, $\overline{Tx}=X$.  We say that $(X,T)$ is {\em
   equicontinuous} when the action of $T$ (considered as a collection of
 self-maps on $X$) is equicontinuous.  It is well-known that this is the case if
 and only if the metric on $X$ can be chosen invariant under the action of $T$.
 Some of the examples constructed in this work show a closely related but less rigid dynamical behaviour
 referred to as \emph{mean equicontinuity}.
 We refer the reader to the literature (e.g., \cite{Fomin1951, Auslander1959,LiTuYe2015,DownarowiczGlasner2015IsomorphicExtensionsAndMeanEquicontinuity,FuhrmannGrogerLenz2018}) 
 for more information on mean equicontinuous systems.
 
 Given a topological dynamical system $(X,T)$, a Borel probability measure $\mu$ on $X$ is called \emph{$T$-invariant}
 if $\mu(\cdot)=\mu(t\, \cdot)$ for all $t\in T$.
 Recall that two $T$-invariant measures $\mu_1$ and $\mu_2$ on $X$ coincide if and only if
 $\int_X\!f\,d\mu_1=\int_X\!f\,d\mu_2$ for every $f$ from the set $\cC(X)$ of
 continuous real-valued functions on $X$.
 An invariant measure is called \emph{ergodic} if for all measurable sets $A\ssq X$ with $t A=A$ ($t\in T$)
 we have $\mu(A)\in \{0,1\}$.
 $(X,T)$ is called \emph{uniquely ergodic} if there exists exactly one invariant measure $\mu$.
 Note that in this case, the unique invariant measure $\mu$ is ergodic.
 
 Suppose $(X,T)$ and $(H,T)$ are topological dynamical systems. Then $(H,T)$
 is called a {\em factor} of $(X,T)$ if there exists a continuous onto map
 $\beta:X\to H$ such that $\beta(tx)=t\beta(x)$ for all $t\in
 T$. The map $\beta$ is called a {\em factor map} in this situation and the preimages of singletons under 
 $\beta$ are referred to as its \emph{fibres}.
 If $\beta$ is bijective, it is called an \emph{isomorphy} and $(X,T)$ and $(H,T)$ are said to be \emph{isomorphic}.
 It is well-known that factor maps preserve minimality and unique ergodicity.

 Given a topological dynamical system $(X,T)$, the {\em Ellis semigroup}
 $E(X)$ associated to $(X,T)$ is defined as the closure of 
 $\{x\mapsto tx \mid t\in T\}\ssq X^X$ in the product topology, where the (semi-)group operation is given by the composition.
 On $E(X)$, we may consider the $T$-action given by $E(X)\ni\tau\mapsto t\tau$ for each element $t\in T$.
 
\begin{thm}[{\cite[pp.~52--53]{auslander1988minimal}}]\label{thm: representation of general equicont sys}
 Suppose $H$ is a compact metric space and $(H,T)$ is minimal and equicontinuous.
 Then $E(H)$ is a compact metrisable topological group.
 Further, we have the following.
 \begin{enumerate}[(a)]
  \item If $T$ is abelian, then $E(H)$ is abelian and $(H,T)$ is isomorphic to $(E(H),T)$.
 \item For general $T$, $(H,T)$ is a factor of $(E(H),T)$, where the factor map $\pi$ is given by
 \[
  \pi\:E(H)\to H,\qquad \tau\mapsto \tau h
 \]
 for some fixed $h\in H$.
 In particular, $\pi$ is open.
 \end{enumerate}
\end{thm}

\begin{rem}
 Throughout this article, abelian groups are always denoted as additive groups,
 whereas general (possibly non-commutative) groups are multiplicative.
\end{rem}

A topological dynamical system $(H,T)$ is called a {\em maximal equicontinuous
factor (MEF)} of $(X,T)$ if it is an equicontinuous factor of $(X,T)$ with
the additional property that every other equicontinuous factor of $(X,T)$ is also a factor of $(H,T)$.

 \begin{lem}[{\cite[p.\ 125, Theorem 1]{auslander1988minimal}}]\label{l.MEF}
   Suppose $(X,T)$ is a topological dynamical system with $X$ compact metric.
   Then $(X,T)$ has a unique (up to conjugacy) MEF $(H,T)$ with $H$ compact
   metric.
 \end{lem}


 Given metric spaces $X$ and $H$, a continuous map $\beta:X\to H$ is called {\em
   almost one-to-one} if
 \[
 X_0 \ = \ \{ x\in X\mid \beta^{-1}(\{\beta(x)\}) = \{x\} \}
 \]
 is dense in $X$. Points in $X_0$ are called {\em injectivity points} of
 $\beta$. If $\beta$ is an almost one-to-one factor map between topological
 dynamical systems $(X,T)$ and $(H,T)$, then $(X,T)$ is called an {\em almost
  one-to-one extension} of $(H,T)$.
It is easy to see that the sets $X_0$ and $\beta(X_0)$ are
residual subsets of $X$ and $H$, respectively.
Moreover, observe that if $(H,T)$ is minimal, 
then $(X,T)$ is also minimal.

 The following elementary fact about almost one-to-one maps will be
 useful. Recall that a compact set $W\ssq X$ is called {\em proper} if
 $\overline{\inte(W)}=W$.
 \begin{lem}
   \label{l.proper_subsets} Suppose that $X$ and $H$ are metric spaces and
   $\beta:X\to H$ is almost one-to-one. Then images of proper subsets of $X$ under
   $\beta$ are proper subsets of $H$.
 \end{lem}
 Suppose now that $(H,T)$ is equicontinuous and minimal. 
 Then $(H,T)$ is uniquely ergodic with a unique $T$-invariant measure $\mu$.
 If $T$ is abelian, 
 we may assume $H=E(H)$ and obtain $\mu=\Theta_H$ where $\Theta_H$ denotes the Haar measure on $H$.
 In general, $\mu$ equals $\Theta_{E(H)}\circ\pi^{-1}$, where $\pi$ is as in 
 Theorem~\ref{thm: representation of general equicont sys}
 and $\Theta_{E(H)}$ denotes the (left) Haar measure on $E(H)$ (which, since $E(H)$ is compact, coincides
 with the right Haar measure \cite[Theorem~15.13]{HewittRoss1963}). 
 In both cases, an almost one-to-one extension
 $(X,T)$ of $(H,T)$ is called {\em regular}, if the projection
 $H_0=\beta(X_0)$ of the set of injectivity points of $\beta$ has positive $\mu$-measure. 
 Otherwise, it is called {\em irregular}. In the regular case,
 $(X,T)$ is uniquely ergodic and measure-theoretically isomorphic to $(H,T)$
 (with respect to the unique invariant measure). Clearly, by ergodicity of $\mu$, the set $H_0$ has full measure in this case.
 We call a group action $(X,T)$ {\em almost automorphic} if it is an almost one-to-one
 extension of its MEF and the MEF is minimal which, by the above observations, implies that $(X,T)$
 is minimal, too.
 \medskip
 
 A system $(X,T)$ is called {\em tame} if the cardinality of its Ellis semigroup
 $E(X)$ is at most $2^{\aleph_0}$, and {\em non-tame} or {\em wild} otherwise
 \cite{Koehler1995EnvelopingSemigroups,Glasner2006OnTameSystems}. A structure
 theorem for minimal tame systems has been established in
 \cite{Glasner2018MinimalTameSystems} (see also
 \cite{Huang2006TameSystemsAndScrambledPairs,Glasner2006TameSystems,KerrLi2007Independence}).
 If $(X,T)$ allows for an invariant measure, it simplifies to the following
 statement.
 \begin{thm}[{\cite[Corollary~5.4]{Glasner2018MinimalTameSystems}}]
   \label{t.tame_almost_automorphic}
   Suppose $(X,T)$ is a minimal and tame group action which has an invariant
   probability measure.  Then $(X,T)$ is an almost one-to-one extension of its
   maximal equicontinuous factor.
 \end{thm}

 The natural question we will focus on is whether this extension is regular or
 not \cite[Problem~5.7]{Glasner2018MinimalTameSystems}.  (It is, as stated in
 Theorem~\ref{t.tame_implies_regular_for_group_actions} and proved in
 Section~\ref{GroupActions} below.) To that end, the following equivalent
 characterisation of tameness will be useful. We call a pair of closed and
 disjoint subsets $U_0,U_1\ssq X$ an {\em independence pair} if there exists an
 infinite set $S\ssq T$ such that for all $a\in\{0,1\}^S$ there is some $\xi\in
 X$ with
 \[
 t\xi \ \in \ U_{a_t} \qquad (t\in S).
 \]
 \begin{thm}[{\cite[Proposition 6.4]{KerrLi2007Independence}}] \label{t.independence_pair}
   A topological dynamical system $(X,T)$ is non-tame if
   and only if there exists an independence pair.
 \end{thm}

 In the case of symbolic dynamics (with the left shift denoted by $\sigma:\{0,1\}^\Z\to\{0,1\}^\Z$), there is the following immediate
 \begin{cor} \label{c.nontameness_subshifts}
   Suppose $\Sigma\ssq\{0,1\}^\Z$ is a subshift (that is, closed and
   shift-invariant) and there exists an infinite set $S\ssq \Z$ such that for
   every $a\in\{0,1\}^S$ there is some $\xi\in\Sigma$ with $\xi_s=a_s$ for all
   $s\in S$. Then $(\Sigma,\sigma)$ is non-tame.
 \end{cor}
 \proof Let $U_0=[0]$ and $U_1=[1]$ be the cylinder sets of length one (at position $0$) in
 $\{0,1\}^\Z$. 
 By the assumptions, $(U_0,U_1)$ forms an independence pair.\qed\medskip

 A very similar statement holds in the case of model sets (see Corollary~\ref{c.nontameness_modelsets} below).

\subsection{Topological entropy}
 \label{subsec:topologicalEntropy}
In the following, $T$ denotes a non-compact, locally compact second countable abelian group with Haar measure $\Theta_T$.
Let $(A_n)_{n\in\N}$ be a van Hove sequence in $T$, that is, $(A_n)_{n\in\N}$ is an exhausting sequence of
 relatively compact subsets of $T$ such that $\overline{A_n} \ssq A_{n+1}$ and for every
 compact $K \ssq T$ we have
 $$ \nLim \frac{1}{\Theta_T(A_n)} \Theta_T(\partial^K (A_n)) = 0,$$ where
 $\partial^K(A_n) = ((K+A_n) \setminus \interior(A_n)) \cup ( (\overline{T
 \setminus A_n}+K) \cap A_n)$ is the {\em $K$-boundary of $A_n$}.\footnote{Observe that every van Hove
 sequence is also a \emph{F\o lner sequence} (as defined in \cite{Lindenstrauss2001}, for example).}
 We say the van Hove sequence is {\em tempered} if there exists $C \geq 1$ such that for all $n \in \N$ the estimate 
 $\Theta_T \left( \bigcup_{k=1}^{n-1} A_k^{-1}A_n \right) \leq C\Theta_T(A_n)$ holds. 
 It is worth mentioning that every van Hove sequence admits a 
 tempered subsequence (see \cite[Proposition~1.4]{Lindenstrauss2001}).

 In the following definitions, we keep the dependence on $(A_n)_{n\in\N}$
 implicit.  Given a topological dynamical system $(X,T,\phi)$, with $(X,d)$ a
 compact metric space, and $C\ssq X$, we say that a set $S \ssq X$ is {\em
   $(\epsilon,n)$-spanning for $C$} if for every $\zeta \in C$
 there is some $\xi \in S$ such that
  $$\max_{s \in A_n} d(\phi(s,\zeta), \phi(s,\xi)) < \epsilon.$$ 
  We denote the minimal cardinality of a set which
 $(\epsilon,n)$-spans $C$ by $S^C(\phi,\epsilon,n)$.  The {\em topological
   entropy of $\phi$ on $C$} is defined as
  $$ \htop^C(\phi) = \lim_{\epsilon \rightarrow 0} h^C_\epsilon(\phi),$$
 where
  $$ h^C_\epsilon(\phi) = \limsup_{n \rightarrow \infty} \frac{1}{\Theta_T(A_n)}
 \log S^C(\phi,\epsilon,n).$$ We set $\htop(\phi) =
 \htop^X(\phi)$.  Suppose $\psi$ is another continuous $T$-action on some
 compact metric space $H$ which is a factor of $(X,\phi)$ with a factor map $\beta$.
 Then $\htop(\phi) \geq
 \htop(\psi)$. For $\xi \in H$, we let $\htop^\xi(\phi) =
 \htop^{\beta^{-1}(\xi)}(\phi)$. Clearly, we obtain $\htop^\xi(\phi) \leq
 \htop(\phi)$ for any $\xi \in H$.  In case of $T = \R$ we have
 \begin{thm}[{\cite[Theorem 17]{bowen:EntropyForGroupEndomorphismsAndHomogeneousSpaces}}]
  \label{thm:entropyBowen}
  If $\phi$ is an $\R$-action, then $\htop(\phi) \leq \htop(\psi) + \sup_{\xi
    \in H} \htop^\xi(\phi).$
 \end{thm}
 If $\htop(\psi) = 0$, the preceding inequalities yield
 $\htop(\phi) = \sup_{\xi \in H} \htop^\xi(\phi)$, that is, positive entropy of $\phi$
 must be realised in single fibres of $\beta$ already.  Note also that, in the
 Euclidean case, vanishing entropy with respect to one van Hove sequence implies
 vanishing entropy with respect to all van Hove sequences (compare
 \cite{baake2007purePointDiffraction}).

\begin{rem}
Extensions of Theorem~\ref{thm:entropyBowen} to $\R^N$-actions and actions of more general groups will be provided in
 \cite{hauser}.
 Even for higher dimensional Euclidean CPS (see Section~\ref{subsec:torusParametrisation} for the definitions), we can therefore
 restrict to showing zero topological fibre entropy, in order to the reduce the technicalities (see Section~\ref{sec: higher dimensional CPS and zero entropy}). 
\end{rem}
 
\subsection{Delone dynamical systems}\label{sec: delone dynamical systems}

Let $G$ be a non-compact, locally compact second countable abelian group with Haar measure
$\Theta_G$. 
Note that, by the Birkhoff-Kakutani Theorem, $G$ is metrisable with a metric $d$ which can be chosen to be
invariant under translations on the group. Furthermore, open balls with respect $d$ 
are relatively compact.  A set $\Gamma \ssq G$ is called {\em ($r$)-uniformly discrete} if there exists $r>0$ such that $d(g,h)>r$ for all $g\neq
h \in\Gamma$. Further, $\Gamma$ is called {\em ($R$)-relatively dense} (or {\em
  syndetic}) if there exists $R>0$ such that $\Gamma\cap B_R(g)\neq \emptyset$
for all $g \in G$, where $B_R(g)$ denotes the $R$-ball centred at $g$. 
We call $\Gamma$ a {\em Delone set} if it is uniformly discrete and relatively dense.
  
Given $\rho>0$ and $g\in\Gamma$, the tuple $(B_\rho(0)\cap (\Gamma-g),\rho)$ is
called a ($\rho$-)patch of $\Gamma$. 
The set of all patches of $\Gamma$ is denoted by
$\cP(\Gamma)$.
A Delone set $\Gamma$ is said to have {\em finite local complexity} (FLC) if for
all $\rho>0$ the number of $\rho$-patches that occur is finite.  Let $\cD(G)$
denote the space of Delone subsets of a given metrisable group $G$.  
Given
$\Gamma, \Gamma' \in \cD(G)$, set
  $$ \dist(\Gamma, \Gamma') = \inf \left\{ \varepsilon > 0 \mid \,\exists \delta \in
 B_\varepsilon(0) \: \Gamma - \delta \cap B_{1/\varepsilon}(0) = \Gamma' \cap
 B_{1/\varepsilon}(0) \right\}. $$ Then $d(\Gamma, \Gamma') = \min \{ 1/\sqrt{2},
 \dist(\Gamma,\Gamma')\}$ defines a metric on $\cD(G)$
 (see \cite[Section~2]{LeeMoodySolomyak2002} as well as \cite[Remark~2.10~(ii)]{MullerRichard2013}).

If $\cD_{r,R}(G)\ssq \cD(G)$ denotes the subset of Delone sets that are
$r$-uniformly discrete and $R$-relatively dense with fixed $r,R>0$, then
$(\cD_{r,R}(G),d)$ is a compact metric space. 
Clearly, $\cD_{r,R}(G)$ is
invariant under the translation action $ \varphi: (g,\Gamma) \mapsto \Gamma-g$.
It follows that for any Delone set $\Gamma\ssq G$ with FLC, the dynamical hull 
$\Omega(\Gamma) = \closure\left(\left\{\Gamma-g\mid g \in G\right\}\right)$ is
compact (\cite[Proposition~2.3]{Schlottmann1999GeneralizedModelSets}). The $G$-action $(\Omega(\Gamma),G)$ given by the translation $\varphi$ is
called a {\em Delone dynamical system}. 
Finally, we will need a criterion for non-tameness analogous to
Corollary~\ref{c.nontameness_subshifts}. Suppose that $\Omega$ is a
translation-invariant subset of $\cD(G)$. We say $\Omega$ {\em admits an
  (infinite) free set} or {\em (infinite) independence set} $S\ssq G$ if there exists a
uniformly discrete set $\Lambda\ssq G$ with $S\ssq \Lambda$ such
that for all $P\ssq S$ there exists some $\Gamma\in \Omega$ with $\Gamma\ssq
\Lambda$ and $\Gamma\cap S=P$.

 \begin{cor} \label{c.nontameness_modelsets}
   Suppose $G$ is a locally compact second countable group, $\Omega\ssq \cD(G)$
   is compact and translation-invariant and $(\Omega,G)$ denotes the action of
   $G$ on $\Omega$ by translations. If $\Omega$ admits an infinite free set, then
   $(\Omega,G)$ is non-tame.
 \end{cor}
\begin{proof} 
Suppose $S\ssq\Lambda$ are as above and $\Lambda$ is $r$-uniformly discrete for some $r>0$. 
Set $U_0=\{\Gamma\in\Omega\mid \Gamma\cap
B_r(0) =\emptyset\}$ and $U_1=\{\Gamma\in\Omega\mid \Gamma \cap
\overline{B_{r/2}(0)}\neq \emptyset\}$.
By the assumptions, $U_0$ and $U_1$ form an independence pair.
\end{proof}

 \subsection{Cut and project schemes and the torus parametrisation}
 \label{subsec:torusParametrisation}

 We refer to standard references such as
 \cite{Meyer1972AlgebraicNumbers,Schlottmann1999GeneralizedModelSets,Moody2000ModelSetsSurvey,BaakeLenzMoody2007Characterization,BaakeGrimm2013AperiodicOrder,lee2006characterization}
 for the following basic facts.  A {\em cut and project scheme} (CPS) consists
 of a triple $(G,H,\cL)$ of two locally compact abelian groups $G$ (called {\em
   external group}) and $H$ ({\em internal group}) and a co-compact discrete
 subgroup $\cL \ssq G \times H$ such that the natural projections $\pi_G : G
 \times H \to G$ and $\pi_H : G \times H \to H$ satisfy
 \begin{itemize}
  \item[(i)] the restriction $\pi_G \vert_\cL$ is injective;
  \item[(ii)] the image $\pi_H(\cL)$ is dense.
 \end{itemize}
 If (i) and (ii) hold, we call $\cL$ an {\em irrational lattice}.  As a
 consequence of (i), if we let $L = \pi_G(\cL)$ and $L^* = \pi_H(\cL)$, the
 {\em star map}
  $$ * : L \rightarrow L^* : l \mapsto l^* =
 \pi_H\circ \left.\pi_{G}\right|_\cL^{-1}(l) $$ is well-defined and
 surjective. 
 Given a compact set $W \ssq H$ (referred to as {\em window}), we define the point set
 \begin{equation*} 
  \oplam(W) \ = \ \pi_G\left(\cL \cap (G\times W)\right) \ = \ \{l\in L\mid
  l^*\in W\}.
 \end{equation*}
 Since $W$ is compact, $\oplam(W)$ is uniformly discrete and has FLC, and if
 $W$ has non-empty interior, then $\oplam(W)$ is relatively dense. Hence, if $W$
 is {\em proper}, $\oplam(W)$ is Delone and has FLC.  In this case, we call
 $\oplam(W)$ a {\em model set}.  The window (and also the resulting model set)
 is called {\em regular} if $\Theta_H(\partial W)=0$, otherwise it is called
 {\em irregular}.  We say a subset $W \ssq H$ is {\em irredundant}\footnote{This
   notion corresponds to that of {\em separating covers} in the context of
   almost automorphic symbolic dynamical systems \cite{Paul1976} or, in the non-symbolic case, to the
   concept of \emph{being invariant under no rotation} in the setting of
   semi-cocycle extensions (see, for example, \cite[Section~5]{DownarowiczDurand2002}).}  if $ \{ h
 \in H \mid h+W=W \} = \{0\}$.  Note that if $\partial W$ is irredundant, then
 $W$ is irredundant, too.  A CPS is called {\em Euclidean} if $G =\R^N$ and $H =
 \R^M$ for some $M,N\in\N$, and {\em planar} if $N=M=1$.
 
 Since $\cL$ is a lattice in $G \times H$, the quotient $\T = (G \times H) /
 \cL$ is a compact abelian group. A natural $G$-action on $\T$ is given by
 $\omega : (u,[s,t]_\cL) \mapsto [s+u,t]_\cL$.  Here, $[s,t]_\cL$ denotes the
 equivalence class of $(s,t) \in G \times H$ in $\T$.  Observe that due to the
 assumptions on $(G,H,\cL)$, this action is minimal.  Further, if the window $W
 \ssq H$ is irredundant, $(\T,G)$ is the maximal equicontinuous factor of the
 Delone dynamical system $(\Omega(\oplam(W)),G)$.  The respective factor map
 $\beta$ is also referred to as {\em torus parametrisation}.
  
 Given an irredundant window $W$, the fibres of the torus parametrisation are
 characterised as follows: For $\Gamma \in \Omega(\oplam(W))$, we have
 \begin{equation}
  \label{eq:flowMorphism}
  \Gamma\in\beta^{-1}([s,t]_\cL) \quad \equi \quad \oplam(\inte(W)+t)-s \ssq
  \ \Gamma \ \ssq \oplam(W+t)-s
 \end{equation} 
 as well as
 \begin{align}
  \begin{split}
  \label{eq:equivalenceForFibre}
  \Gamma \in \beta^{-1}([0,t]_\cL) \quad \equi \quad& \exists\, (t_j) \in
         {L^*}^\N \text{ with } \jLim t_j = t \text{ and } \jLim \oplam(W+t_j) =
         \Gamma.
  \end{split}
 \end{align} 
 In particular, if $(\partial W + t) \cap L^* = \emptyset$, we have
 $\oplam(\inte(W)+t)-s = \oplam(W+t)-s$ so that $\beta^{-1}([s,t]_\cL)$ is a
 singleton for each $s$.  We denote the set of such $t\in H$ by $\mc G_W$, that is, 
\[
   \mc G_W \ = \ -\left(H\smin \bigcup_{\ell^*\in L^*} \partial W-\ell^*\right) \ . 
\] 
 Note that $\mc G_W$ is residual.  Hence, $(\Omega(\oplam(W+t)-s),G)$ is an
 almost automorphic system if $t\in \mc G_W$.  In this case, an (ir)regular
 model set $(\Omega(\oplam(W+t)-s),G)$ is an (ir)regular extension of $(\T,G)$.
 Thus, the present notion of (ir)regularity is consistent with that of
 Section~\ref{TopDyn}.
 
  \begin{rem}
  \label{remark: inclusion in fibres}
  Let $\Gamma, \Gamma' \in \beta^{-1}([s,t]_\cL)$ and $\varepsilon > 0$. Recall
  that $d(\Gamma,\Gamma') < \varepsilon$ if and only if $\Gamma$ and $\Gamma'$
  coincide on $B_{1/\varepsilon}(0)$ up to a small translation $\delta \in
  B_\varepsilon(0)$.  Observe that $\oplam(W+t)-s$ is $r$-uniformly discrete
  with $r=\min_{l \neq l' \in \oplam(W+t)} d_G(l,l')$.  Hence, if $\varepsilon <
  \frac{r}{2}$, equation (\ref{eq:flowMorphism}) immediately yields $\delta =
  0$, i.e., $d(\Gamma, \Gamma') < \varepsilon'$ if and only if $\Gamma \cap
  B_{1/\varepsilon'}(0) = \Gamma' \cap B_{1/\varepsilon'}(0)$.
 \end{rem}

 \begin{rem}
  Note that if $W$ is not irredundant, it is possible to construct a CPS $(G,H',\cL')$ with irredundant window 
  $W' \ssq H'$ such that for each $\Lambda \in \Omega(\oplam(W))$ with $\oplam(\interior(W)) \ssq \Lambda \ssq \oplam(W)$ we have $\oplam(\interior(W')) \ssq \Lambda \ssq \oplam(W')$
  (compare \cite[Section~5]{lee2006characterization} and \cite[Lemma~7]{BaakeLenzMoody2007Characterization}).
  Thus, in the following, we may assume that all occurring windows are irredundant.
 \end{rem}

 
 \section{Tame implies regular}
  \label{TameImpliesRegular}
  
 The aim of this section is to prove that tame implies regular for model sets,
 minimal subshifts and general minimal topological actions.  The result on subshifts
 will be obtained as a special case of the one for model sets, whereas the case
 of general group actions requires some modifications to adapt the proof to the
 more general setting. We should also point out that the information obtained
 for irregular model sets is slightly stronger than in the general case, since
 the existence of an infinite free set --as defined in Section~\ref{sec: delone dynamical systems}-- is stronger than the existence of an
 independence pair.

 \subsection{The case of model sets}

\begin{thm} \label{t.irregular_implies_free_set}
  Suppose that $(G,H,\cL)$ is a CPS with locally compact and second countable
  abelian groups $G,H$. Denote by $\Theta_H$ the Haar measure on $H$. If $W$ is
  a proper window with $\Theta_H(\partial W)>0$, or if $H=\kreis (=\R/\Z)$ and
  $\partial W$ is a Cantor set, then $\Omega(\oplam(W))$ admits an infinite free
  set $S\ssq G$.
\end{thm}

The proof will be based on the following criterion for the existence of infinite
free sets, which translates the dynamical problem into a purely geometric
question about the structure of the window.

\begin{lem}\label{lem: free_set_criterion}
  Suppose that $(G,H,\cL)$ is a CPS that satisfies the above assumptions and
  there exists a relatively compact set $S^*\ssq L^*$ such that for each
  $P^*\ssq S^*$ we have
  \begin{equation} \label{e.window_translates_CPS}
     H(S^*,P^*) \ = \ \bigcap_{s^*\in P^*} (W-s^*)\ \cap\! \bigcap_{s^*\in S^*\smin P^*}\!(W^c-s^*)\
     \cap\ (-\cG_W) \ \neq \ \emptyset.
  \end{equation}
  Then the set $S=\{s\in G\mid s^*\in S^*\}$ is free.
\end{lem}
\proof Let $P\ssq S$ and choose $h\in H(S^*,P^*)$. 
Then $\oplam(W-h)\in \Omega(\oplam(W))$ since $-h\in \cG_W$ (see Section~\ref{subsec:torusParametrisation}). 
Further, for any $s\in S$, we have
\[
s\in P \ \equi \ h\in W-s^* \ \equi \ s^*\in W-h \ \equi \ s\in \oplam(W-h).
\]
Therefore, $\oplam(W-h)\cap S=P$. 
Note that if $P^*\neq \emptyset$, then $H(S^*,P^*)\ssq W-P^*\ssq W-S^*$.
Hence, we may assume without loss of generality that the points
$h$ from above belong to the compact set $V:=\overline{W-S^*}\cup \{h_0\} \ssq H$,
where $h_0$ is some point in $(W^c-S^*)\cap (-\cG_W)$.
Clearly, $\oplam(W-h)\ssq
\oplam(W-V)=:\Lambda$ for $h\in V$.
Since $W-V$ is a compact window, the set
$\Lambda$ is uniformly discrete. As $P\ssq S$ was arbitrary, we obtain that
$S\ssq \Lambda$ is a free set.\qed \medskip

Hence, in order to prove Theorem~\ref{t.irregular_implies_free_set}, it suffices
to prove the existence of an infinite set $S^*$ that satisfies the assumptions of the previous
lemma. 
With a view towards the later extension of the proof to topological group actions
in Section \ref{GroupActions}, we reformulate the required statement in a
slightly more abstract form.
%

\begin{prop} \label{p.window_translates}
  Suppose that $H$ is a locally compact second countable Hausdorff topological
  group with left Haar measure $\Theta_H$ and $V_0,V_1\ssq H$ are closed subsets that satisfy
  \begin{itemize}
  \item[(i)] $\overline{\inte(V_0)}=V_0$ and $\overline{\inte(V_1)}=V_1$,
  \item[(ii)] $\inte(V_0)\cap \inte(V_1)=\emptyset$,
  \item[(iii)] $\Theta_H(V_0\cap V_1)>0$.
  \end{itemize}
  Further, assume that $T\ssq H$ is a dense subgroup and $\cG\ssq H$ is a
  residual set.  Then there exists an infinite set $I\ssq T$ such that for all
  $a\in\{0,1\}^I$ there exists $h\in \cG$ with the property that
  \begin{equation}
    \label{e.existence_free_set}
       th \in \inte(V_{a_t}) \qquad (t\in I).
  \end{equation}
  The same result holds if $H=\kreis$ and (iii) is replaced by the assumption
  that $V_0\cap V_1$ is a Cantor set.
\end{prop}

\begin{rem}
In the situation of Theorem~\ref{t.irregular_implies_free_set}, we can apply
this statement with $V_1=W$, $V_0=\overline{W^c}$, $T=L^*$ and
$\cG=-\cG_W$. This yields an infinite set $I$ that satisfies the assertions of
the proposition.  Moreover, for $a_t=1$ ($t\in I$), equation
(\ref{e.existence_free_set}) yields $h\in \cG$ with $h+I\ssq V_1=W$ which, by
the compactness of $W$, gives that $I$ is relatively compact.  Hence, $S^*=I$
satisfies the assumptions of Lemma~\ref{lem: free_set_criterion}, and this
proves Theorem~\ref{t.irregular_implies_free_set}. (Note that $h$ in
(\ref{e.existence_free_set}) is contained in the respective intersection in
(\ref{e.window_translates_CPS}) with $P^*=\{s^*\in S^*\mid a_{s^*}=1\}$.)
\end{rem}

For the proof of the proposition, we need some measure-theoretic estimates
concerning intersections of translates of $V_0\cap V_1$.
We denote the right Haar-measure on $H$ by $\rH$.
Recall that $\rH$ (as well as the left Haar measure $\Theta_H$) 
on a locally compact second countable group $H$ is
outer regular. 
Hence, if $C\ssq H$ is a Borel set of positive measure and we set
\[
\eta^C(\eps)=\frac{\rH(B_\eps(C))}{\rH(C)} - 1,
\]
then $\lim_{\eps\to 0} \eta^C(\eps) = 0$.

Let $\Sigma_n=\{0,1\}^n$ (with $n\in \N$) and $\Sigma_*=\bigcup_{n\in \N} \Sigma_n$. 
Denote by $|a|$
the length of a word $a\in \Sigma_*$ (so that $a\in\Sigma_{|a|}$ for all
$a\in\Sigma_*$).
In the following, we assume the metric $d$ on the group $H$ to be invariant 
under multiplication from the left.
We denote the neutral element of $H$ by $e$.

\begin{lem} \label{lem: measure_estimate}
  Suppose that $C\ssq H$ is a Borel set with $\rH(C)>0$ and
  $(\xi_a)_{a\in\Sigma_*}$ is a family of elements $\xi_a\in H$. Let
  $(\eps_n)_{n\in\N}$ be a sequence of positive real numbers such that 
  \[
  \eps_n \ \geq \ \sup_{a\in\Sigma_n} d(e,\xi_a).
  \]
  For $j\in\N,\ n\in\N\cup\{\infty\}$, let $\delta_j^n=\sum_{\ell=j}^n \eps_\ell$.
  Further, given $n\in\N$ and $a\in\Sigma_n$, let $\gamma_a=\prod_{j=1}^n
  \xi_{a_1\ld a_j}=\xi_{a_1}\xi_{a_1,a_2}\ldots\xi_{a_1\ld a_n}$. Then for each $n\in\N$, we have
  \begin{equation}
    \label{e.measure_estimate}
    \rH\left(\bigcap_{a\in\Sigma_n} C\gamma_a^{-1}\right) \ \geq \ \rH(C)\cdot
    \left(1-\sum_{j=1}^n 2^{j-1}\eta^C(\delta_j^n)\right).
  \end{equation}
\end{lem}
\proof We proceed by induction on $n$.\smallskip

 \noindent{\em Base case $(n=1)$:}\quad Note that $\Sigma_1=\{0,1\}$,
 $\gamma_0=\xi_0,\ \gamma_1=\xi_1$, $\delta_1^1=\eps_1$ and
 $d(e,\xi_0^{-1}),d(e,\xi_1^{-1})<\eps_1$.  Both $C\gamma_0^{-1}$ and
 $C\gamma_1^{-1}$ are sets of measure $\rH(C)$ that are contained in
 $B_{\delta_1^1}(C)$ which is of measure $(1+\eta^C(\delta^1_1))\cdot
 \rH(C)$. This implies that $\rH((C\gamma_0^{-1})\cap (C\gamma_1^{-1})) \geq
 \rH(C)\cdot (1-\eta^C(\delta^1_1))$ as required.\smallskip

\noindent{\em Inductive step $(n\to n+1)$:} \quad Suppose the statement holds for
some $n\in\N$, all sets $C\ssq H$, and all collections $(\xi_a)_{a\in\Sigma_*}$ as well as all sequences $(\eps_n)$ as
above. Given $a\in\Sigma_*$, let $\xi_a'=\xi_{0a}$ and $\xi_a''=\xi_{1a}$ and
define $\gamma_a',\gamma_a''$ accordingly. Then
\[
 \bigcap_{a\in\Sigma_{n+1}} C\gamma_a^{-1} \ =
 \ \Bigg(\underbrace{\left(\bigcap_{a'\in\Sigma_n}
   C{\gamma'_{a'}}^{-1}\right)}_{=: I'} \xi_0^{-1}\Bigg) \ \cap \  \Bigg(\underbrace{\left(\bigcap_{a''\in\Sigma_n}
   C{\gamma''_{a''}}^{-1}\right)}_{=: I''} \xi_1^{-1}\Bigg).
 \]
 By the induction hypothesis, both $I'\xi_0^{-1}$ and $I''\xi_1^{-1}$ are sets of measure
 \[
    \rH(I'),\rH(I'') \ \geq \ \rH(C)\cdot\left( 1-\sum_{j=1}^n
    2^{j-1}\eta^C(\delta_{j+1}^{n+1})\right)
    \]
    that are contained in $B_{\delta_1^{n+1}}(C)$. Hence, we obtain that
     \begin{align*}
    \rH\left((I'\xi_0^{-1})\cap (I''\xi_1^{-1})\right) & \geq  \rH(C)\cdot \left(1-
    \eta^C(\delta^{n+1}_1) - 2\cdot\sum_{j=1}^n 2^{j-1}\eta^C(\delta_{j+1}^{n+1})
    \right) \\ & = \rH(C)\cdot
    \left(1-\sum_{j=1}^{n+1} 2^{j-1}\eta^C(\delta_j^{n+1})\right).
    \end{align*}
    This completes the proof.  \qed\medskip

We can now turn to the \proof[\textbf{\textit Proof of
    Proposition~\ref{p.window_translates}}.]

Let $\cG=\bigcap_{n\in\N} G_n$, where each $G_n$ is an open and dense subset of
$H$. We will construct a sequence $(t_n)_{n\in\N}$ of points in $T$ and a
collection $(U_a)_{a\in\Sigma_*}$ of compact subsets
of $H$ with the following properties for all
$n\in\N$ and $a\in \Sigma_n$
\begin{itemize}
\item[{\bf (I1)}] $U_a\ssq (t_n^{-1}\, \inte(V_0))\cap G_n$ if $a_n=0$ and
  $U_a\ssq (t_n^{-1}\, \inte(V_1))\cap G_n$ if $a_n=1$;
\item[{\bf (I2)}] $U_{a0}\cup U_{a1}\ssq U_a$.
\end{itemize}
This will prove the statement: if we define
$I=\{t_n\mid n\in\N\}$, then for any given $a\in\{0,1\}^I$ we let
$a^{(n)}=(a_{t_1}\ld a_{t_n})$ and obtain from (I2) that $\bigcap_{n\in\N}
U_{a^{(n)}}$ is a nested intersection of compact sets and therefore
non-empty. 
By (I1), any $h\in \bigcap_{n\in\N} U_{a^{(n)}}$ has the property that $h\in\cG$ and $t_n h\in \inte(V_{a_{t_n}})$ for all $n\in\N$, as
required by (\ref{e.existence_free_set}).
\medskip

We are going to construct $(t_n)_{n\in\N}$ and $(U_a)_{a\in\Sigma_*}$ by
induction on $n=|a|$.  Let us first specify some details.  We will choose $U_a$
as closed balls of the form $U_a=\overline{B_{r(|a|)}(\gamma_a)}$ which we can
ensure to be compact by choosing $r(n)$ sufficiently small.  We set
$\xi_0=\gamma_0$, $\xi_1=\gamma_1$ and
$\xi_{a0}=\gamma_a^{-1}\gamma_{a0},\ \xi_{a1}=\gamma_{a}^{-1}\gamma_{a1}$ for
$a\in\Sigma_n$, $n\geq 1$.  By definition, we hence have
$\gamma_a=\prod_{j=1}^n\xi_{a_1\ld a_j}$ which is consistent with the notation
of Lemma~\ref{lem: measure_estimate}.  Further, we let $C=V_0\cap V_1$.  Observe
that if $\Theta_H(C)>0$, then $\rH(C)>0$ since $\Theta_H$ and $\rH$ are mutually absolutely
continuous \cite[Theorem~15.15 \& Comment~15.27]{HewittRoss1963}.  We fix a sequence $(\eps_n)_{n\in\N}$ such
that
\[
     \sum_{j=1}^\infty 2^{j-1} \eta^C(\delta^\infty_j) \ < \ 1,
\]
where the $\delta_j^n$ are defined as in
Lemma~\ref{lem: measure_estimate}. We moreover include the condition
\begin{itemize}
\item[{\bf (I3)}] $\sup_{a\in\Sigma_n} d(e,\xi_a) \leq \eps_n$
\end{itemize}
in the inductive assumption.  Note that this boils down to choosing
$\gamma_{a0}$ and $\gamma_{a1}$ $\eps_{n+1}$-close to $\gamma_a$ in each step of
the construction. \smallskip

In the case that $H=\kreis$ and $C=V_0\cap V_1$ is a Cantor set, the sequence
\nfolge{\eps_n} and condition (I3) will not be needed. Instead, we will use the
assumption
\begin{itemize}
\item[{\bf (I3')}] for all $a\in\Sigma_n$ we have $\partial U_a\ssq t_{n+1}^{-1}C$ 
\end{itemize}
in this case.

Let us first consider the case that $\Theta_H(C)>0$.
\smallskip

{\em Base case $(n=1)$:}\quad We choose $t_1\in T$ and two open balls
$U_0'=\overline{B_{r(1)}(\xi_0')}\ssq \inte(V_0)\cap t_1 (G_1\cap
B_{\eps_1}(e))$ and $U_1'=\overline{B_{r(1)}(\xi_1')}\ssq \inte(V_1)\cap t_1
(G_1\cap B_{\eps_1}(e))$.
If we let $U_0=t_1^{-1}U'_0$ and $U_1=t_1^{-1}U_1'$, then (I1) and (I3) are
satisfied for $n=1$, and (I2) is still void.\smallskip

{\em Inductive step $(n\to n+1)$:}\quad Suppose now that
$t_1\ld t_n$ and $U_a$ for $a\in \bigcup_{j=1}^n \Sigma_j$ have been chosen and
satisfy (I1)--(I3).  Then Lemma~\ref{lem: measure_estimate} gives
\[
\rH\left(\bigcap_{a\in\Sigma_n} C\gamma_a^{-1}\right)
\ \geq \ \rH(C)\cdot \left(1-\sum_{j=1}^n 2^{j-1} \eta^C(\delta^\infty_j)\right) \ > \ 0.
\]
In particular, the set on the left is non-empty and we can choose $h\in
\bigcap_{a\in\Sigma_n} C\gamma_a^{-1}$. 
Clearly, $\gamma_a\in
h^{-1}C$ for all $a\in\Sigma_n$. 
Now, we choose 
$t_{n+1}\in T$ close enough to $h$ to guarantee
that $t_{n+1}^{-1}C$ intersects $B_{r'(n+1)/2}(\gamma_a)\ssq U_a$ for all
$a\in \Sigma_n$, where $r'(n+1)=\min\{\eps_{n+1},r(n)\}$. 
However, since points in $C$ lie in the closure of the interior
of both $V_0$ and $V_1$, this allows to find $r(n+1)>0$ as well as closed
balls $U_{a0}=B_{r(n+1)}(\gamma_{a0})$ and $U_{a1}=B_{r(n+1)}(\gamma_{a1})$ with
midpoints $\gamma_{a0}$ and $\gamma_{a1}$ $\eps_{n+1}$-close to $\gamma_a$ for
all $a\in \Sigma_n$ such that (I1)--(I3) are satisfied for $n+1$.

If $H=\kreis$ and $C$ is a Cantor set, the statement follows in a similar way
without invoking Lemma~\ref{lem: measure_estimate}. The crucial observation here
is that if we choose some $\Delta t_{n+1}$ sufficiently close to zero, then the rotation by
$t_{n+1}=\Delta t_{n+1}t_{n}$ will send one of the endpoints of each $U_a$, $a\in\Sigma_n$, into
$\inte(U_a)$ (the left endpoints if $\Delta t_{n+1}$ is locally to the right of zero
and vice versa). Hence, we arrive at a similar situation as in the first
case. \qed\medskip

Altogether, we have now completed the proof of
Theorem~\ref{t.irregular_implies_free_set}.

\subsection{Application to symbolic systems} \label{SymbolicSystems}
In the following, given a subshift $(\Sigma,\mathbb{Z})$, we denote by 
$\beta\: \Sigma\to H$ the factor map onto its MEF.
Note that $(H,\Z)$ is completely characterised by a self-map on $H$ which we denote by $\rho$, that is,
$nh=\rho^n(h)$ ($h\in H,\, n\in \N$).

The basis for the direct application of the results from the last section to symbolic systems is
provided by the following fact. 
 
 \begin{prop}[Compare \cite{BaakeJaegerLenz2016ToeplitzFlowsModelSets,Paul1976}]
   An almost automorphic subshift $(\Sigma,\mathbb{Z})$ is isomorphic to the
   system $(\Omega(\oplam(W)),\mathbb{Z})$ obtained from the CPS
   $(\mathbb{Z},H,\mathcal{L})$ with lattice $\mathcal{L}=\{(n,\rho^n(h_0))\:
   n\in\mathbb{Z}\}$, where
   \begin{itemize}
   \item $h_0\in H$ has unique preimage under the factor map $\beta$;
   \item $W=\beta([1])$, where $[1] = \{\xi \in\{0,1\}^\mathbb{Z}\: \xi_0=1\}$.
   \end{itemize}
   Moreover, the window $W$ is proper, that is,
   $\overline{\mathrm{int}(W)}=W$.
 \end{prop}

 As a consequence of Theorem~\ref{t.irregular_implies_free_set} and Corollary
 \ref{c.nontameness_subshifts}, we obtain
 \begin{cor}
   If an almost automorphic subshift $(\Sigma,\Z)$ is irregular, then it has an
   infinite free set. In particular, it is non-tame. The same result holds if
   the maximal equicontinuous factor is an irrational circle rotation and
   $\beta([0])\cap\beta([1])$ is a Cantor set.
 \end{cor}

 For the special case of Toeplitz flows (where $(H,\Z)$ is an adding machine),
 a similar result has been established previously
 by Downarowicz \cite{Downarowicz}.
 Further, note that the existence of an infinite
 free set also implies positive sequence entropy.
 
\subsection{The case of minimal group actions} \label{GroupActions}

Theorem~\ref{t.tame_implies_regular_for_group_actions} provides an analogue to
Theorem~\ref{t.irregular_implies_free_set} for the case of general automorphic systems. 
However, it is worth noting that it does not
imply Theorem~\ref{t.irregular_implies_free_set} as a corollary, since the
existence of an infinite free set --as defined in Section~\ref{sec: delone dynamical systems}-- does not follow directly from
non-tameness.

%

\begin{proof}[Proof of Theorem~\ref{t.tame_implies_regular_for_group_actions}]
Let us denote the maximal equicontinuous factor of $(X,T)$ by $(H,T)$.
As $(H,T)$ is minimal and equicontinuous, Theorem~\ref{thm: representation of general equicont sys} implies 
that $(H,T)$ is a factor of $(E(H),T)$.
We denote the corresponding factor map by $\pi$ and the unique $T$-invariant measure on $H$ by $\mu$.
Recall that $\mu=\Theta_{E(H)} \circ \pi^{-1}$ and $\pi$ is open.

Assume for a contradiction that $\beta$ is not almost surely one-to-one with respect to the
measure $\mu$ on $H$ so that $(X,T)$ is an irregular extension of
$(H,T)$.  
We aim to show the existence of an independence pair $(U_0,U_1)$ for
$(X,T)$ which implies non-tameness by Theorem~\ref{t.independence_pair}.

To that end, denote by $\cK(X)$ the space of compact subsets of $X$, equipped
with the Hausdorff metric $d_H$, and consider the mapping
\[
F : H\to \cK(X), \qquad \xi \mapsto \beta^{-1}(\xi).
\]
By compactness of $X$ and continuity of $\beta$, the map $F$ is upper
semicontinuous and hence measurable.  By Lusin's theorem, we may therefore
choose a compact set $K\ssq H$ of positive measure such that $F_{|K}$ is
continuous.  Let $K_0\ssq K$ denote the topological support of the measure
$\mu|_K$ (that is, the essential closure of $K$). Then, $\mu(K_0)=\mu(K)>0$.
Hence, by irregularity, we can find $h_0\in K_0$ such that $\sharp
\beta^{-1}(h_0)>1$.  Moreover, we have that $\mu(V\cap K)>0$ for any
neighbourhood $V$ of $h_0$.

Choose $\xi_0\neq \xi_1\in \beta^{-1}(h_0)$ and let $\eps=d(\xi_0,\xi_1)/4$ and
$U_0=\overline{B_\eps(\xi_0)},\ U_1=\overline{B_\eps(\xi_1)}$. 
We aim to show that
$(U_0,U_1)$ is an independence pair for $(X,T)$, that is, there is an infinite
set $I\ssq T$ such that for any $a\in\{0,1\}^I$ there exists $\xi\in X$ with the
property that
\begin{equation}
  \label{e.independent_orbits}
  t\xi \in U_{a_t} \qquad (t\in I).
\end{equation}
Let $V_0=\beta(U_0)$ and $V_1=\beta(U_1)$. 
By Lemma~\ref{l.proper_subsets}, both
these sets are proper, that is, $V_0=\overline{\inte(V_0)}$ and
$V_1=\overline{\inte(V_1)}$. Moreover, they have disjoint interiors since
points with singleton fibres are dense.

Due to the continuity of $F$ on $K$, we can choose $\delta>0$ such that for any
$h\in B_\delta(h_0)\cap K$ we have $d_H(F(h),F(h_0))<\eps$.  This yields that
the fibre $F(h)=\beta^{-1}(h)$ intersects both $U_0$ and $U_1$, so that $h\in
V_0\cap V_1$.  Therefore, $B_\delta(h_0)\cap K\ssq V_0\cap V_1$ so that
$\mu(V_0\cap V_1)\geq \mu(B_\delta(h_0)\cap K)>0$.  Set $V_0'=\pi^{-1}(V_0)$
and $V_1'=\pi^{-1}(V_1)$.  Since $\pi$ is open, $V_0'$ and $V_1'$
are proper.
Moreover, $\Theta_{E(H)}(V_0'\cap V_1')=\mu(V_0\cap V_1)>0$.  Thus, the assertions of
Proposition~\ref{p.window_translates} are met by $V_0',V_1'\ssq E(H)$
with $\cG=\pi^{-1}(\beta(X_0))$, where $X_0$ denotes the set of injectivity
points of $\beta$ (observe that $\cG$ is residual, since $\pi$ is open).

Hence, we obtain an infinite set $I\ssq T$, and for each $a\in\{0,1\}^I$ a point $h'\in
\cG$ such that $t h'  \ \in \ \inte(V_{a_s}')$ $(t\in I)$ and hence
\begin{equation} \label{e.independent_base_orbit}
  t h  \ \in \ \inte(V_{a_s}) \qquad (t\in I)
\end{equation}
for $h=\pi(h')\in \beta(X_0)$.
However, since $h$ has a unique preimage under $\beta$ (and the same is true for
all points in its orbit), (\ref{e.independent_base_orbit}) directly implies
(\ref{e.independent_orbits}) so that $(U_0,U_1)$ is an independence pair as
claimed.  
\end{proof}

\section{Self-similarity and locally disjoint complements: two criteria for zero entropy} 
 \label{Self-similarity}
 Let $G$ and $H$ be locally compact abelian second countable groups and let $G$
 be non-compact.  Consider a CPS $(G,H,\cL)$ with proper window $W \ssq H$ and
 torus parametrisation $\beta : \Omega(\oplam(W)) \to \T$.  In this section, we
 provide sufficient criteria for zero entropy of $(\Omega(\oplam(W)),G)$ in
 terms of the local structure of $W$.  We note that all points $[s,t]_\cL \in
 \T$ are translates of $[0,t]_\cL$ and hence $\sharp \beta^{-1}([s,t]_\cL) =
 \sharp \beta^{-1}([0,t]_\cL)$.  Therefore, in the following, it is sufficient
 to consider points $[0,t]_\cL \in \T$.

\subsection{Self similar windows.} 
 \label{self similar windows} 
 As a direct consequence of the discussions in Section \ref{subsec:topologicalEntropy} and  
 Theorem \ref{thm:entropyBowen}, we obtain 
 
 \begin{lem}
  \label{lem:zeroEntropy}
  \begin{enumerate}[(i)]
   \item  If $\sharp \beta^{-1}(\xi) < \infty$ for $\xi \in \T$, then $\htop^\xi(\varphi)=0$. 
   \item  If $G = \R$ and $\sharp \beta^{-1}(\xi) < \infty$ for all $\xi \in \T$, then $\htop(\varphi) = 0$. 
  \end{enumerate}
 \end{lem}
 In view of the above lemma, it is our first goal to control the number of elements in the fibres of
 $\beta$. Note that the above assumptions are quite strong, since {\em all} fibres are assumed
 to be finite.
 Even if the measure of $\partial W$ vanishes (so that $\sharp \beta^{-1}(\xi) = 1$ for $\Theta_\T$-a.e. $\xi \in \T$), there
 may still exist fibres with infinite cardinality
 (compare the constructions in
 \cite{JaegerLenzOertel2016PositiveEntropyModelSets},
 \cite{MarkleyPaul1979}). 

 Consider a point $[0,t]_\cL \in \T$. As a consequence of equation
 \eqref{eq:flowMorphism}, Delone sets contained in $\beta^{-1}([0,t]_\cL)$
 basically differ from each other in points $l$ whose conjugates $l^*$ are
 contained in $\partial W + t$. Hence, in case of $(\partial W + t) \cap L^* =
 \emptyset$, the fibre $\beta^{-1}([0,t]_\cL)$ is a singleton, carrying no
 entropy. In case $(\partial W + t) \cap L^* \neq \emptyset$, the cardinality of
 $\{ l^* \mid l^* \in (\partial W + t) \cap L^* \}$ may be finite or not.  In the
 first case, the cardinality of the respective fibre of $\xi=[0,t]_\cL$ will be
 finite, so that we immediately obtain zero entropy by Lemma
 \ref{lem:zeroEntropy}.  Hence, we just have to investigate the latter
 case. Note that by Birkhoff's Ergodic Theorem, positive measure of $\partial W$
 ensures the existence of points $[0,t]_\cL \in \T$ such that $\sharp \{ l^*
 \mid l^* \in (\partial W + t) \cap L^* \} = \infty$.

 We call $[s,t)_\cL$ \emph{critical} if $(\partial W+t)\cap L^*\neq \emptyset$. 
 For a given critical point $[0,t]_\cL$, we say that $l_1^*, l_2^* \in
 \partial W + t \cap L^*$ are {\em similar with respect to $t$} if there exists
 some $\epsilon > 0$ such that
  $$ (B_{\epsilon}(l_1^*) \cap (W + t)) - l_1^* = (B_{\epsilon}(l_2^*) \cap (W +
 t)) - l_2^*. $$ Clearly, being similar with respect to some fixed $t$ is an
 equivalence relation. We call the corresponding equivalence classes {\em
   similarity classes with respect to $t$}. If for each $t$ there are only
 finitely many similarity classes, we call {\em $W$ self similar}. If the
 maximal number of similarity classes for any $t$ is $k$, we call $W$ {\em
   $k$-self similar}. Finally, if $k=1$, then we call $W$ {\em perfectly self
   similar}.
  
 \begin{lem}
  \label{lem: self similar points yield similar sets} 
  Let $[0,t]_\cL \in \T$. Suppose $l_1^*, l_2^* \in (\partial W + t) \cap L^*$ are
  similar with respect to $t$. Then for each $\Gamma \in \beta^{-1}([0,t]_\cL)$,
  we have $l_1 \in \Gamma$ if and only if $l_2 \in \Gamma$.
 \end{lem}
 \begin{proof}
  Let $\Gamma \in \beta^{-1}([0,t]_\cL)$ and suppose $l_1 \in \Gamma$. Due to
  Equation \ref{eq:equivalenceForFibre} and FLC, there exists a sequence $t_j \in
  L^*$ with $t_j\to t$ such that $l_1 \in \oplam(W+t_j)$ for all
  $j$. Thus, $l_1^* \in W+t_j$. By the assumptions, we also
  obtain $l_2^* \in W+t_j$ for large enough $j$. Hence, $l_2 \in \jLim
  \oplam(W+t_j) = \Gamma$.  
  By symmetry, the statement follows.
 \end{proof}
  
 \begin{lem}
  Let $(G,H,\cL)$ be a CPS with proper window $W \ssq H$. If there exist $k$
  similarity classes with respect to $t$, then $\beta^{-1}([0,t]_\cL)\leq
  2^k$. Hence, if $W$ is self similar, we have $\sharp \beta^{-1}(\xi)<\infty$
  and therefore $\htop^\xi(\varphi) = 0$ for all $\xi \in \T$, and if $W$ is
  $k$-self similar, then $\beta^{-1}(\xi)\leq 2^k$ for all $\xi\in\T$.
 \end{lem}
 Note that in particular this means that if $W$ is perfectly self similar, then
 all fibres contain at most two elements, so that $(\Omega(\Lambda(W)),G)$ is a
 $2$-$1$-extension of $(\T,G)$.
 \begin{proof}
  Fix an arbitrary $\xi = [0,t]_\cL \in \T$. Without loss of generality, we may assume $\xi$
  to be critical. By the self similarity of $W$, there are finitely many
  equivalence classes $E_1^*, \ldots, E_p^*$ such that $(\partial W + t) \cap L^*
  = \bigcup_{i=1}^p E_i^*$.
  Due to Lemma \ref{lem: self similar points yield similar sets}, 
  $\beta^{-1}(\xi)<2^p$.  
  By Lemma~\ref{lem:zeroEntropy}, we obtain $\htop^\xi(\varphi) = 0$.
 \end{proof}
 
 \begin{rem}\label{rem: two-to-one extension with upper and lower graph}
  A perfectly self similar window $W\ssq \R$ (for arbitrary planar CPS
  $(\R,\R,\cL)$) will be constructed in Section~\ref{sec: filling the gaps} (see
  Lemma~\ref{lem: self-similarity of W}). 
 \end{rem}

 \subsection{Windows with locally disjoint complements.}
 In the above considerations, we obtained zero entropy by ensuring that
 $\beta$ has finite fibres.  However, under certain assumptions on $W$, the
 entropy vanishes although the fibres contain infinitely many
 elements. 
 We say $W$ has {\em locally disjoint complements} if for all 
 critical $[0,t]_\cL \in \T$ and $l_1^*, l_2^* \in
 (\partial W + t) \cap L^*$ there exists $\epsilon > 0$ such that
  $$ \left( \left( B_\epsilon(l_1^*) \cap (W+t)^c\right)-l_1^*\right) \cap
 \left( \left( B_\epsilon(l_2^*) \cap (W+t)^c \right)-l_2^* \right) =
 \emptyset. $$
  
 \begin{lem}
  \label{lem:characterisationOfLDCWindows}
  Suppose $(G,H,\cL)$ is a CPS with proper window $W \ssq H$ and $W$ has locally disjoint complements.
  Then for all critical $\xi = [0,t]_\cL \in \T$ there is $\Gamma_+\in \beta^{-1}(\xi)$ such that
  for all $\Gamma\in \beta^{-1}(\xi)$ we have that
  \begin{enumerate}[(i)]
  \begin{multicols}{2}
   \item $\Gamma \ssq \Gamma_+$,
   \item $\Gamma$ differs from $\Gamma_+$ in at most one point.
  \end{multicols}
  \end{enumerate}
 \end{lem}
 \begin{proof}
  Fix a critical $\xi = [0,t]_\cL \in \T$ and let $l_0^* \in \partial W +t \cap
  L^*$. 
  Due to equations \eqref{eq:flowMorphism} and \eqref{eq:equivalenceForFibre}, there exists $\Gamma' \in
  \beta^{-1}(\xi)$ such that $l_0 \notin \Gamma'$ and a sequence $t_j' \rightarrow t$
  such that $\Gamma' = \jLim \oplam(W+t_j')$ and $l_0^* \in
  (W+t_j')^c$.  
  Now, let $l^* \in (\partial W + t \cap
  L^*) \setminus \{l_0^*\}$.
  Since $W$ has locally disjoint complements, there exists $\epsilon > 0$ such that 
   $$0 \in (B_\epsilon(l_0^*) \cap (W+t_j')^c) - l_0^* \Longrightarrow 0
  \notin (B_\epsilon(l^*) \cap (W+t_j')^c)-l^* $$ for large enough $j$.
  Hence, $l^* \in W+t_j'$ for sufficiently large $j$ which implies $l \in
  \Gamma'$.
  
  As $l_0^*$ was arbitrary, the above yields the existence of a
  sequence $(\Gamma_n)$ in $\beta^{-1}(\xi)$ such that $\{ l \mid l^* \in
  \partial W + t \cap L^*\} \cap B_n \ssq \Gamma_n$.  Compactness of
  $\beta^{-1}(\xi)$ gives a convergent subsequence with limit $\Gamma_+$ which
  verifies (i). (ii) follows immediately.
 \end{proof}
 
 \begin{lem}
  \label{lem:zeroEntropyForFatFibres}
  Let $(G,H,\cL)$ be a CPS with proper window $W \ssq H$. If $W$ has locally
  disjoint complements, we have $\htop^\xi(\varphi) = 0$ for all $\xi \in \T$.
 \end{lem}
 \begin{proof}
  Let $\xi = [0,t]_\cL \in \T$ be critical and $(A_n)$ a van Hove sequence in
  $G$.  By Lemma \ref{lem:characterisationOfLDCWindows}, there is $\Gamma_+ \in
  \beta^{-1}(\xi)$ such that every other set $\Gamma \in \beta^{-1}(\xi)$
  differs from $\Gamma_+$ in one point. We denote this point by $l(\Gamma)$.
  For $\epsilon > 0$ and $n \in \N$, define
   \begin{align*}
   S(\varphi,\epsilon,n) &= \{ \Gamma \in \beta^{-1}(\xi) \mid l(\Gamma)
  \in K_{1/\epsilon} +A_n\} \cup \{\Gamma_+\}\\
  &=\{ \Gamma \in \beta^{-1}(\xi) \mid l(\Gamma)
  \in \partial^{K_{1/\epsilon}}(A_n) \cup A_n \} \cup \{\Gamma_+\},
  \end{align*}
  where $K_{r}$ denotes the closed $r$-ball around $0$.
  Observe that $\beta^{-1}(\xi)=  \bigcup_{\Gamma \in S(\varphi,\epsilon,n)} \{
  \Gamma' \in \beta^{-1}(\xi) \mid \max_{s \in A_n}
  d(\varphi_s(\Gamma),\varphi_s(\Gamma')) < \epsilon \}$.  
  Thus,
  $S(\varphi,\epsilon,n)$ is an $(\epsilon,n)$-spanning set for $\beta^{-1}(\xi)$. Further,
   $$ \sharp S(\varphi,\epsilon,n) \leq \frac{1}{\Theta_G(K_r)} \Theta_G
  \left( \partial^{K_{1/\epsilon+r}}(A_n) \cup A_n \right),$$ where $r = \min_{l
    \neq l' \in \Gamma_+} d_G(l,l')/2$. This yields
  \begin{align*}
   h_\epsilon^\xi(\varphi) &\leq \limsup_{n \to \infty}
   \frac{1}{\Theta_G(A_n)} \log \left( \frac{1}{\Theta_G(K_r)} \Theta_G
   \left( \partial^{K_{1/\epsilon+r}}(A_n) \cup A_n \right) \right) \\ &\leq
   \limsup_{n \to \infty} \frac{1}{\Theta_G(A_n)} \log \left(
   \frac{1}{\Theta_G(K_r)} \right) + \frac{1}{\Theta_G(A_n)} \left( \log
   \Theta_G \left(\partial^{K_{1/\epsilon+r}}(A_n) \right) + \log \Theta_G(A_n)
   \right).
  \end{align*}
  Since $A_n$ is a van Hove sequence, we clearly have $\nLim \frac{1}{\Theta_G(A_n)}
  \log \Theta_G(\partial^{K_{1/\epsilon+r}}(A_n)) = 0$ which yields
  $h_\epsilon^\xi(\varphi) = 0$.  Hence, $\htop^\xi(\varphi) = 0$.
 \end{proof}
 
 Let us point out that, in fact, Lemma~\ref{lem:zeroEntropyForFatFibres} 
 readily follows from the next statement if the variational principle holds 
 for continuous actions of $G$ on compact metric spaces.
 
 \begin{lem}\label{lem: locally disjoint complements implies unique ergodicity}
  Let $(G,H,\cL)$ be a CPS with proper window $W \ssq H$ and suppose $W$ has locally
  disjoint complements.
  Then $(\Omega(\oplam(W)),G)$ is uniquely ergodic.
 \end{lem}
 \begin{proof}
%
%
%
  Let $\nfolge{A_n}$ be a tempered van Hove sequence in $G$ and suppose there exist two invariant ergodic measures
  $\mu_1,\mu_2$ on $\Omega(\oplam(W))$.
  Given $f\in \mathcal C(\Omega(\oplam(W)))$ and $i\in\{1,2\}$, 
  Lindenstrauss' Pointwise Ergodic Theorem 
  \cite[Theorem~1.2]{Lindenstrauss2001} yields
  a subset $\Omega_i^f \ssq \Omega(\oplam(W))$ of full $\mu_i$-measure such that 
  for all $\Gamma \in \Omega_i^f$
  \begin{align}\label{eq: Birkhoff average equals space average}
   A_n(f,\Gamma)\coloneqq \frac{1}{\Theta_G(A_n)} \int_{A_n} f(\Gamma-s) \ \mathrm{d}\Theta_G(s) 
  \stackrel{n\to\infty}{\xrightarrow{\hspace*{1.2cm}}}
  \mu_i(f)\coloneqq\int_{\Omega(\oplam(W))}\!f \,d\mu_i.
  \end{align}
  
  We want to show that \eqref{eq: Birkhoff average equals space average} holds
  for all $\Gamma\in\beta^{-1}(M_i^f)$, where $M_i^f=\beta(\Omega_i^f)$.
  To that end, given $g_0\in G$ and $\eps>0$, consider
  \[f_{g_0;\eps}\: \Omega(\oplam(W))\to \R, \quad f_{g_0;\eps}(\Gamma)=
   \max\left\{0,\, 1-1/\eps\cdot \min_{g\in \Gamma}d(g,B_\eps(g_0))\right\}.
   \]
   Clearly, 
   $f_{g_0;\eps}$ is continuous and the set $\cF=\{f_{g_0;\eps}\: g_0\in G,\, \eps>0\}$ separates points
   and contains the constant function equal to $1$.
    Its algebraic closure $\cF'$ is hence dense in $\cC(\Omega(\oplam(W)))$, due to the Stone-Weierstrass Theorem.

    Now, given $f\in \cF$,
   Lemma~\ref{lem:characterisationOfLDCWindows} immediately yields
   \begin{align}\label{eq: equality of Birkhoff averages within fibres}
     \nLim A_n(f,\Gamma') =\nLim A_n(f,\Gamma) \qquad(\Gamma \in \Omega_i^f(\oplam(W)),\, \Gamma'\in\beta^{-1}(\beta(\Gamma))).
   \end{align}
   Observe that \eqref{eq: equality of Birkhoff averages within fibres} straightforwardly extends to all $f\in \cF'$ and 
   thereby, in fact,
   to all $f\in \overline{\cF'}=\cC(\Omega(\oplam(W)))$.
   This shows \eqref{eq: Birkhoff average equals space average} for all $\Gamma\in \beta^{-1}(M_i^f)$ with 
   $f\in \cC(\Omega(\oplam(W)))$.
  
  Since $\beta$ sends $\mu_1$ and $\mu_2$ to the unique invariant measure $\Theta_\T$ on $\T$, we clearly have $M_1^f\cap M_2^f\neq \emptyset$.
  Hence, $\mu_1(f)=\lim_{n\to\infty} A_n(f,\Gamma)=\mu_2(f)$ 
  for all $\Gamma \in \beta^{-1}(M_1^f\cap M_2^f)$.
  Since $f\in \cC(\Omega(\oplam(W)))$ was arbitrary, this shows $\mu_1=\mu_2$ and thus finishes the proof. 
 \end{proof}
 
 \begin{rem}
  In Section~\ref{sec: filling the gaps}, we construct a planar CPS with window $V \ssq \R$ with
  locally disjoint complements.  
  For an (implicit) application of the
  criterion of locally disjoint complements outside the setting of Euclidean CPS, see
  \cite[Example~5.1]{DownarowiczKasjan2015}.
\end{rem}

\section{Construction of a self-similar window for planar CPS} \label{MainConstruction}
The main goal of this section is to show that for each \emph{planar} CPS $(\R,\R,\mc L)$, there are irredundant
windows with boundaries of positive Lebesgue measure which are 
self-similar and which have locally disjoint complements.
By means of the results of the previous section, this proves Theorem~\ref{t.counterexamples}.
Moreover, given any \emph{higher dimensional} CPS $(\R^n,\R,\mc L)$, we show that there are
windows such that the associated Delone dynamical system has zero topological entropy.

\subsection{Planar CPS and irrational rotations.}
For the constructions in this section, it is important to note that the set $L^*$ is generated 
by an irrational circle rotation: observe that for each irrational lattice 
$\cL \ssq \R^2$ there exists a matrix $A = \begin{pmatrix} a_{11} & a_{12} \\ a_{21} & a_{22} \end{pmatrix} \in \GL(2,\R)$ 
with $a_{11}/a_{12}, a_{21}/a_{22} \in \R \setminus \Q$ such that $\cL = A(\Z^2)$.
Put $\omega= a_{21}$. 
Without loss of generality, we may assume $a_{22} = 1$. Thus 
 $$ L^* = \pi_2(\cL) = \left\{ n\omega+m \mid (n,m) \in \Z^2 \right\} = \pi^{-1}( \{ n\omega \mod 1 \mid n \in \Z \}), $$
where $\pi : \R \to\T^1$ denotes the canonical projection onto $\T^1= \R / \Z$ and $\pi_2\:\R^2\to \R$ is the projection to the second coordinate.

As seen in Section~\ref{Self-similarity}, the entropy of the Delone dynamical system $(\Omega(W),\R)$
is related to the local structure of $W+t$ at points in $L^* \cap \partial W + t$. 
Given $t \in \R$, if $W \ssq [0,1]$, then a point in $L^* \cap \partial W+t$ corresponds to some $n\in \Z$ with $n \omega-t \mod 1 \in \partial W$.
Thus, a self-similar window $W\ssq [0,1]$ for the CPS $(\R,\R,\mc L)$ can be understood as
a subset $W\ssq \T^1$ such that for all orbits $\mc O(x)=x+\omega \Z$ (of the rotation on $\T^1$ by angle $\omega$)
there are finitely many $n_1,\ldots,n_N\in \Z$ such that for all $y=x+n\omega\in \partial W\cap \mc O(x)$ 
there is $i\in \{1,\ldots,N\}$ and $\eps>0$ with
$(B_\eps(y)\cap W)+(n_i-n)\omega=B_\eps(x+n_i\omega)\cap W$. 
Consistently with the terminology in Section~\ref{Self-similarity}, we call a subset of $\T^1$ with this property \emph{self-similar}.

In the following, we fix $\w \in \T^1\setminus \Q$ and denote
by $\f$ the rotation by $\w$ on $\T^1$, that is, $\f(x)=x+\w$.
Without loss of generality, we may assume $|\w|<1/2$.
We set $q_1=\min\{\ell\in\N\:d(\f^\ell(0),0)<|\w|\}$
and define the sequence of \emph{closest return times} $(q_n)_{n\in\N}$ recursively by putting
$q_{n+1}=\min\{\ell\in\N\:d(\f^\ell(0),0)<d(\f^{q_n}(0),0)\}$. 
We further set $I_n$ to be the closed interval of length $|I_n|=d(\f^{q_n}(0),0)$ with endpoints 
$0$ and $\f^{q_n}(0)$.
Our construction makes use of the following well-known facts (see, e.g., \cite{Swierczkowski1959}, \cite[Chapter~I.1]{deMelovanStrien1993} and \cite[Theorem~4.5]{Kurka2003}).
\begin{prop}\label{prop: properties of partition by closest return intervals}
 Given an irrational rotation $\f$ on $\T^1$, 
 let $\mc P_n=\{\f^j(I_n)\:1\leq j\leq q_{n+1}\}\cup \{\f^j(I_{n+1})\: 1\leq j\leq q_n\}$,
 where $q_n$ and $I_n$ are defined as above.
 Then
 \begin{enumerate}[(i)]
  \item $\T^1=\bigcup_{J\in \mc P_n} J$ and $\mathring{J_1}\cap \mathring{J_2}=\emptyset$ for each $J_1\neq J_2\in \mc P_n$;
  \item For each $J\in\mc P_n$ and each $m>n$, there is $\mc Q_{J,m} \ssq \mc P_{m}$
  such that $J=\bigcup_{K\in \mc Q_{J,m}}K$;
  \item If $J,J'\in \mc P_n$ and $J =\f^\ell(J')$ for some $\ell\in \N$, then
  $\mc Q_{J,m}=\f^\ell(\mc Q_{J',m})(=\{\f^\ell(K)\: K\in \mc Q_{J',m}\})$ for all $m> n$.
  \end{enumerate}
\end{prop}
\begin{rem}
 In simple terms, (i) gives that the elements of $\mc P_n$ basically partition $\T^1$ and (ii) yields
 that the partition by elements of $\mc P_{n+1}$ is a refinement of that given by $\mc P_n$.
 Point (iii) is to be understood as a self-similarity of the respective partitions.
\end{rem}

\begin{rem}
Our goal is to construct a proper set $W$ (which we want to be self-similar) and another proper set
$V$ (with locally disjoint complements) whose boundaries are irredundant and verify
$\textrm{Leb}_{\T^1}(\partial W),\, \textrm{Leb}_{\T^1}(\partial V)>1-\eps$
for $0<\eps<1$.

To that end, we first construct an irredundant self-similar Cantor set $C$ as the limit
of a nested sequence $(C_\ell)$ of recursively defined compact subsets of $\T^1$.
At each step $\ell$ of the construction, the set $C_{\ell}$ is obtained by removing elements of $\mc P_{n_\ell}$ from $C_{\ell-1}$ 
(for some appropriately chosen increasing sequence $(n_\ell)$) so that by Proposition~\ref{prop: properties of partition by closest return intervals} (ii), $C_{\ell}$ is
a union of intervals from $\mc P_{n_\ell}$.
To establish the self-similarity of the limit set $C$, we treat
the intervals which comprise $C_{\ell-1}$ equally.
That is, roughly speaking, if $J_1,J_2\in \mc P_{n_{\ell-1}}$ with $J_1,J_2\ssq
C_{\ell-1}$ are translated copies of each other, say $\f^n (J_1)=J_2$, then we
keep $J\in \mc Q_{J_1,n_{\ell}}$ in $C_\ell$ if and only if $\f^n(J)\in \mc
Q_{J_2,n_{\ell}}$ is kept in $C_\ell$.  Eventually, the limit set $C$ will serve
as the boundary of both $W$ and $V$.  We will obtain $W$ by filling the gaps of
$C$ in such a way that the self-similarity of $C$ is preserved while it has to
be destroyed in a particular way in order to obtain $V$.

We note that the simple idea of {\em `treating all partition intervals of equal
  length in $C_\ell$ equally in all subsequent steps`} rather easily leads to self
similar windows with at most two similarity classes for each $t\in\R$ (see Remark~\ref{rem: simpler self-similar boundary} below),
and hence
at most four elements in every fibre. The construction presented below is
somewhat more subtle, as it includes some refinements in order to produce
\emph{perfectly} self similar windows with a Cantor set boundary.
\end{rem}

\subsection{Construction of a self-similar Cantor set}
Given $0<\eps<1$, pick a sequence $(\beta_\ell)$ of positive numbers with
$\sum_{\ell=1}^\infty 3\beta_\ell<\eps$ and let $(n_\ell)$ be a sequence of
positive integers with $|I_{n_\ell+1}|/|I_{n_{\ell+1}}|>1/\beta_\ell$.  For
technical reasons, we may assume without loss of generality
that $n_{\ell+1}\geq n_\ell+6$.
In particular, this yields
$\#Q_{J,n_{\ell+1}}\geq 8$ for each $\ell\in \N$ and $J\in \mc P_{n_\ell}$.

We recursively define a nested sequence of compact sets $(C_\ell)\ssq \T^1$ whose
limit will be a Cantor set $C$ satisfying the desired self-similarity condition.
To that end, let us introduce some terminology.
Suppose we have already constructed $C_\ell\ssq C_{\ell-1}\ssq \ldots \ssq C_1= \T^1$.
We call a connected component $J$ of the complement $C_\ell^c$
a \emph{gap} of $C_\ell$ and we say $J$ is of \emph{level $k$} (with $k\in\{2,3,\ldots \ell\}$) if $J\cap C_k^c\neq \emptyset$ and $J\cap C_{k-1}^c= \emptyset$.
We further call an interval $J\in\mc P_{n_\ell}$ with $J\ssq C_\ell$ \emph{$k$-accessible from the left/right}
if its left/right\footnote{In any of the two, from now on fixed, orientations on $\T^1$.} endpoint is
at the boundary of a gap of $C_\ell$ which is of level $k$.
It is worth mentioning that we will construct $C_\ell$ $(\ell\in \N)$ in such a way that each $J\in \mc P_{n_\ell}$ is 
accessible from at most one side. 
Given $C_\ell$, we obtain $C_{\ell+1}$ by removing from $C_\ell$
\begin{enumerate}[(1)]
 \item the interior of the two left-most/right-most intervals as well as the interior of the right-most/left-most
 interval of $\mc Q_{J,n_{\ell+1}}$ if
 $J\in \mc P_{n_\ell}$ is $k$-accessible from the left/right and $\ell-k$ is even;
 \item the interior of the left-most and the right-most interval
of $\mc Q_{J,n_{\ell+1}}$ for all $J\in \mc P_{n_\ell}$ which haven't been dealt with in (1);
\item all isolated points which remain after having removed intervals according to (1) \& (2).
\end{enumerate}

Put $C=\bigcap_{\ell} C_\ell$.
Observe that $C$ is a Cantor set of positive Lebesgue measure as it is compact, nowhere dense 
(since $\mc O^+(0)=\{\f^n(0)\:n\in \N\}\ssq C^c$),
and
\[
 \Leb_{\T^1}(C)=\lim_{\ell\to\infty}  \Leb_{\T^1}(C_\ell)\geq 1-\sum_{\ell=1}^\infty 3\beta_\ell>1-\eps.
\]
Moreover, it turns out that $C$ is irredundant (see Section~\ref{sec: filling the gaps}).

\begin{rem}\label{rem: simpler self-similar boundary}
 Coming back to the last paragraph of the previous section, 
 observe that in order to provide a self-similar Cantor set, we could
 simply follow step (2) and (3) but this time applying (2) to \emph{every} $J\in \mc P_{n_\ell}$.
 Let us denote the resulting Cantor set of this simplified construction by $\tilde C$.
 In principle, we could replace $C$ by $\tilde C$ in the following. 
 As a matter of fact, this would not change the proofs of some of the next statements 
 (in particular, Lemma~\ref{lem: sim is an equivalence relation} and Lemma~\ref{lem: self similar cantor set})
 while the proof of Lemma~\ref{lem: along one orbit all points on C are equivalent} would even be shortened.
 However, as we point out in Remark~\ref{rem: simpler self-similar window} below, $\tilde C$ can't be the boundary of a 
 perfectly self-similar window, that is,
 there are fibres of the factor map $\beta$ from $(\Omega(\tilde W),\R)$ onto $(\T^1,\f)$
 with more than two elements.
\end{rem}

Before we turn to the construction of the sets $W$ and $V$, let us study $C$ locally along orbits.
Given $x\in\T^1$, $n\in \Z$ and $\ell\in \N$, we write $x\sim_\ell \f^n(x)$ 
whenever 
$x,\f^n(x)\notin \bigcup_{j=1}^{|n|}\f^j(I_{n_\ell}\cup I_{n_{\ell}+1})$
and if there are $1\leq j_0,j_1\leq q_{n_\ell+1-i}$, with $i$ either $0$ or $1$, such that
\begin{enumerate}[(a)]
\item  $x\in \textrm{int}\,\f^{j_0}(I_{n_\ell+i})$ and $\f^n(x)\in \textrm{int}\,\f^{j_1}(I_{n_\ell+i})$;
\item $\f^{j_0}(I_{n_\ell+i})$ and $\f^{j_1}(I_{n_\ell+i})$ are from the same side $k$- and $k'$-accessible, respectively,
with $k-k'$ even or $\f^{j_0}(I_{n_\ell+i})$ and $\f^{j_1}(I_{n_\ell+i})$ are not accessible at all.
\end{enumerate}

In the following, the reader should keep in mind that, by construction, $\mc
O^+(0)\cap C=\emptyset$ so that for each $x\in C$ we have that $x\in
\f^{j_0}(I_{n_\ell+i})$ actually means $x\in
\textrm{int}\,\f^{j_0}(I_{n_\ell+i})$.
\begin{lem}\label{lem: sim is an equivalence relation}
 Consider $x\in C$.
 If $x\sim_\ell\f^n(x)$ for some $\ell\in \N$ and $n\in \Z$ with $\f^n(x)\in C_\ell$, then 
 $\f^n(x)\in C$.
\end{lem}
\begin{proof}
 Let $j_0,j_1$ be as above.
 Observe that $j_0+n=j_1$ because of Proposition~\ref{prop: properties of partition by closest return intervals}~\!(i)
 and because $x,\f^n(x)\notin \bigcup_{j=1}^{|n|}\f^j(I_{n_\ell}\cup I_{n_{\ell}+1})$.
 Hence, the distance of $\f^n(x)$ to the left (and right) endpoint of $\f^{j_1}(I_{n_\ell+i})=\f^{j_0+n}(I_{n_\ell+i})$
 equals the distance of $x$ to the left (and right) endpoint of $\f^{j_0}(I_{n_\ell+i})$.
 By Proposition~\ref{prop: properties of partition by closest return intervals}~\!(iii),
 we further have $\mc Q_{\f^{j_1}(I_{n_\ell+i}),n_{\ell+1}}=\f^n(\mc Q_{\f^{j_0}(I_{n_\ell+i}),n_{\ell+1}})$.
 By definition of $C_{\ell+1}$, this indeed shows $\f^n(x)\in C_{\ell+1}$ as well as $x\sim_{\ell+1} \f^n(x)$ and hence gives
 $\f^n(x)\in C$
 by induction on $\ell$.
%
\end{proof}

The next statement is crucial for establishing the self-similarity of $C$.

\begin{lem}\label{lem: along one orbit all points on C are equivalent}
 If $x\in C$ and $y\in \mc O(x)\cap C$, then $x\sim_\ell y$ for sufficiently large $\ell$.
\end{lem}
\begin{proof}
 Without loss of generality, we may assume $y=\f^{-n}(x)$ for some $n\in \N$.
 As $\mc O^+(0)\cap C=\emptyset$ and due to Proposition~\ref{prop: properties of partition by closest return intervals}~(i),
 there is $\ell_0\in \N$ such that for all $\ell\geq \ell_0$ there is $i_\ell\in \{0,1\}$ with $x \in \textrm{int} \bigcup_{j=2n+1}^{q_{n_\ell+1-i_\ell}}\f^j(I_{n_\ell+i_\ell})$ 
 and hence $y \in \textrm{int}\bigcup_{j=n+1}^{q_{n_\ell+1-i_\ell}}\f^j(I_{n_\ell+i_\ell})$. 
 In other words, there are $1\leq j_0^{\ell},j_1^{\ell}\leq q_{n_\ell+1-i_\ell}$ with $x\in \textrm{int}\,\f^{j_0^{\ell}}(I_{n_\ell+i_\ell})$ and 
 $y\in \textrm{int}\, \f^{j_1^{\ell}}(I_{n_\ell+i_\ell})=\textrm{int}\,\f^{j_0^\ell-n}(I_{n_\ell+i_\ell})$.
 As in the proof of Lemma~\ref{lem: sim is an equivalence relation}, we see that
 $x$ and $y$ have the same distance to the endpoints of $\f^{j_0^\ell}(I_{n_\ell+i_\ell})$ and $\f^{j_1^\ell}(I_{n_\ell+i_\ell})$, respectively,
 and that $\mc Q_{\f^{j_1^\ell}(I_{n_\ell+i_\ell}),n_{\ell+1}}=\f^{-n}(\mc Q_{\f^{j_0^\ell}(I_{n_\ell+i_\ell}),n_{\ell+1}})$.
 It remains to show (b) for sufficiently large $\ell$.
 To that end, pick some $\ell\geq\ell_0$.
 We have to consider the following cases.
 
 
 \emph{Case 1:
 $\f^{j_0^\ell}(I_{n_\ell+i_\ell})$ and $\f^{j_1^\ell}(I_{n_\ell+i_\ell})$ are accessible from different sides.}
 By the construction of $C_{\ell+1}$ and as $\# \mc Q_{\f^{j_0}(I_{n_\ell+i_0}),n_{\ell+1}}\geq 6$, we 
 have that either $\f^{j_0^{\ell+1}}(I_{n_{\ell+1}+i_{\ell+1}})$ and $\f^{j_1^{\ell+1}}(I_{n_{\ell+1}+i_{\ell+1}})$ are accessible from the same side,
 or at least one of the two intervals is not accessible.
 Hence, we have reduced the problem to one of the following cases.
 
 \emph{Case 2: $\f^{j_0^{\ell}}(I_{n_{\ell}+i_{\ell}})$ as well as $\f^{j_1^{\ell}}(I_{n_{\ell}+i_{\ell}})$ are from the same side 
 $k$- and $k'$-accessible, respectively, and $k-k'$ is odd.}
 We may assume without loss of generality that both intervals are accessible from the right and that $\ell-k$ is even.
 If $\f^{j_0^{\ell+1}}(I_{n_{\ell+1}+i_{\ell+1}})$ is still accessible from the right, then 
 $\f^{j_1^{\ell+1}}(I_{n_{\ell+1}+i_{\ell+1}})$ is not accessible anymore since $\f^{j_1^{\ell}}(I_{n_{\ell}+i_{\ell}})$ has been dealt with
 according to (2) while $\f^{j_0^{\ell}}(I_{n_{\ell}+i_{\ell}})$ has been dealt with according (1).
 Hence, we are in Case 4.
 If, however, $\f^{j_0^{\ell+1}}(I_{n_{\ell+1}+i_{\ell+1}})$ is not accessible from the right anymore, then 
 the same is true for $\f^{j_1^{\ell+1}}(I_{n_{\ell+1}+i_{\ell+1}})$ and hence either both are $\ell+1$-accessible from the left
 or not accessible at all.
 In both cases we are done.
 
 \emph{Case 3: $\f^{j_0^{\ell}}(I_{n_{\ell}+i_{\ell}})$ is $k$-accessible 
 (from some side) with $\ell-k$ even while $\f^{j_1^{\ell}}(I_{n_{\ell}+i_{\ell}})$ is not accessible.}
 Without loss of generality, we may assume that $\f^{j_0^{\ell}}(I_{n_{\ell}+i_{\ell}})$ is accessible from the right.
 If $\f^{j_0^{\ell+1}}(I_{n_{\ell+1}+i_{\ell+1}})$ is accessible from the left or not accessible at all,
 the same holds true for $\f^{j_1^{\ell+1}}(I_{n_{\ell+1}+i_{\ell+1}})$ and we are done.
 If $\f^{j_0^{\ell+1}}(I_{n_{\ell+1}+i_{\ell+1}})$ is still $k$-accessible from the right, then
 $\f^{j_1^{\ell+1}}(I_{n_{\ell+1}+i_{\ell+1}})$ is not accessible.
 In this case we are in Case 4.
 
 \emph{Case 4: $\f^{j_0^{\ell}}(I_{n_{\ell}+i_{\ell}})$ is $k$-accessible 
 (from some side) with $\ell-k$ odd while $\f^{j_1^{\ell}}(I_{n_{\ell}+i_{\ell}})$ is not accessible.}
 Without loss of generality, we may assume that $\f^{j_0^{\ell}}(I_{n_{\ell}+i_{\ell}})$ is accessible from the right.
 Note that $\f^{j_0^{\ell+1}}(I_{n_{\ell+1}+i_{\ell+1}})$ is 
 still $k$-accessible from the right if
 and only if $\f^{j_1^{\ell+1}}(I_{n_{\ell+1}+i_{\ell+1}})$
 is $k'$-accessible from the right as well with $k'=\ell+1$ in which case we are done since $k'-k$ is even.
 In the situation where $\f^{j_0^{\ell+1}}(I_{n_{\ell+1}+i_{\ell+1}})$ is accessible from the left or
 not accessible at all, the same holds true for $\f^{j_1^{\ell+1}}(I_{n_{\ell+1}+i_{\ell+1}})$
 and we are done, too.
\end{proof}
%
%

\begin{lem}\label{lem: self similar cantor set}
 Suppose we are given $x,y\in C$ with $y=\f^n(x)$ for some $n\in \Z$. 
 Then there is $\eps>0$ such that $\f^n(B_\eps(x)\cap C)=B_\eps(y)\cap C$.
\end{lem}
\begin{proof}
 Due to Lemma~\ref{lem: along one orbit all points on C are equivalent}, there is $\ell\in\N$ such that $x\sim_\ell y$.
 In particular, there are hence $i\in\{0,1\}$ and $1\leq j_0,j_1\leq q_{n_{\ell}+1-i}$ with
 $x\in \textrm{int}\,\f^{j_0}(I_{n_\ell+i})\ssq C_\ell$ and 
 $\f^n(x)\in \textrm{int}\,\f^{j_1}(I_{n_\ell+i})\ssq C_\ell$.
 Let $\eps>0$ be such that $B_\eps(x)\ssq \f^{j_0}(I_{n_\ell+i})$ (and hence, $B_\eps(y)\ssq \f^{j_1}(I_{n_\ell+i})$, too).
 Suppose there is $z\in B_\eps(x)\cap C$.
 Then we clearly have $\f^n(z)\in C_\ell$ and $z\sim_\ell \f^n(z)$.
 By Lemma~\ref{lem: sim is an equivalence relation}, we hence have $\f^n(z)\in C$.
 In other words, $\f^n(B_\eps(x)\cap C)\ssq B_\eps(y)\cap C$.
 Similarly, we get the opposite inclusion and hence obtain the desired equality.
\end{proof}

\subsection{Filling the gaps of the self-similar Cantor set}\label{sec: filling the gaps}
We now come to the construction of two windows $W$ and $V$ with $\partial W=\partial V=C$ 
which give rise to model sets which are almost automorphic extensions of $(\T,\f)$ and have zero entropy.
These model sets are at two different ends of low complexity dynamics:
the fibres of $\beta\:\Omega(\oplam(W))\to \T$ have at most two elements and
$\Omega(\oplam(W))$ allows for two distinct ergodic measures;
$\beta\:\Omega(\oplam(V))\to \T$ 
is a metric isomorphy (see the discussion in Section~\ref{Spectra}) and hence $(\Omega(\oplam(V)),\R)$ is mean equicontinuous (which, in particular, yields unique ergodicity)
for the prize of fibres of infinite cardinality.
It is worth mentioning that infinite fibres are, in fact, a necessary requirement for an irregular almost automorphic system
to be mean equicontinuous \cite{FuhrmannGrogerLenz2018}.

In order to construct $W$ and $V$, it remains to fill the gaps of $C$, that is, the connected components of
$C^c$, appropriately.
As a preparation, we first take a closer look at the accessible points of $C$.
To that end, let us provide the following observation.
\begin{prop}
  For all $\ell\in \N$, there are $J^1_{\ell+1},J^2_{\ell+1}\in \mc P_{n_{\ell+1}}$ such that
 for each $J\in\mc P_{n_{\ell}}$ the left-most interval (the interval second from left) of $\mc Q_{J,n_{\ell+1}}$
  is a translated copy of $J^1_{\ell+1}$ ($J^2_{\ell+1}$).
  
  A similar statement holds if we replace left by right.
\end{prop}
\begin{proof}
We only consider the ``left case''.
Without loss of generality, we may assume that $I_{n_\ell+2}$ is an interval to the right of zero (otherwise, we may proceed with $n_\ell+3$ instead of $n_\ell+2$).
 Now, recall that $q_{n+1}\geq q_{n}+q_{n-1}$ for each $n\in \N$.\footnote{In fact,
if $a_{n}$ is the $n$-th coefficient of the continued fraction expansion of $\w$, then $q_{n+1}=a_n q_{n}+q_{n-1}$ \cite[Section~4.4]{Kurka2003}.}
  Given any $J\in\mc P_{n_{\ell}}$, this yields that the left-most interval of
 $\mc Q_{J,n_{\ell}+2}$ is a translated copy of $I_{n_\ell+2}$.
 Since we assume $n_{\ell+1}\geq n_\ell+6$, the statement follows by means of Proposition~\ref{prop: properties of partition by closest return intervals}~(iii).
\end{proof}


Now, let $(J_n)$ be an enumeration of the gaps of $C$ and denote by $x_n\in C$ the right endpoint
of the gap $J_n$.
Similarly as in the previous section, we say that $J_n$ is of \emph{level} $\ell$ if $J_n\cap C_\ell^c\neq \emptyset$ while
$J_n\ssq C_{\ell-1}$.
Assuming that $J_n$ is of level $\ell$, let $y_n$ denote the isolated point in $J_n$ which had to be removed
in step (3) of the construction of $C_\ell$.
Now, let $k$ be the level of $J_n$
and $k'$ the level of $J_{n'}$ and assume without loss of generality that $k< k'$.
Suppose $k'-k$ is even.
%
Then 
\begin{align*}
x_n-x_{n'}=y_n-y_{n'}+\sum_{\ell=k}^{k'-1}
\alpha_{-1^{\ell-k};\ell}
\end{align*}
where $\alpha_{1;\ell}=|J_\ell^1|$ and $\alpha_{-1;\ell}=|J_\ell^1|+|J_\ell^2|$ (recall that every second step of the construction of $C_{k'}$, we remove two intervals
on either side of $J_n\cap C_k$).
Hence, $x_n-x_{n'}$ is an integer multiple of $\w$.
In other words, all right endpoints of the even-level gaps of $C$ belong to one orbit
and all right end-points of odd-level gaps belong to one orbit.
Now, suppose $k-k'$ is odd.
Then 
\begin{align*}
x_n-x_{n'}=y_n-y_{n'}+\sum_{\ell=k}^{k'-1}\alpha_{-1^{\ell-k};\ell}+\sum_{\ell=k'}^\infty(-1)^{\ell-k'}|J_\ell^2|.
\end{align*}
We may assume without loss of generality
that $\sum_{\ell=2}^\infty(-1)^{\ell}|J_\ell^2|$ (and thus $\sum_{\ell=k'}^\infty(-1)^{\ell-k'}|J_\ell^2|$)
is not an integer multiple of $\omega$.\footnote{
First, by possibly going over to a subsequence, we may assume that $2\sum_{\ell=k+1}^\infty|J_\ell^2|<|J_k^2|$ for all integers $k\geq 2$.
Hence, $\sum_{\ell=2}^\infty(-1)^{\ell_j}|J_{\ell_j}^2|\neq \sum_{\ell=2}^\infty(-1)^{\ell_j'}|J_{\ell_j'}^2|$
for distinct subsequences $(n_{\ell_j})$ and $(n_{\ell_j'})$ of $(n_\ell)$.
Second, there clearly are uncountably many subsequences but only countably many integer multiples of $\w$.}
Then $x_n$ and $x_{n'}$ belong to different orbits of $\f$.
Similarly, we have that the left endpoints of even-level gaps belong to one orbit and those of
odd-level gaps belong to a different one.
Since two gaps are of equal length if and only if they are of the same level, this gives that $C$ is, in fact, irredundant.

Without loss of generality, we may assume in the following that $J_{2n}$ is of an even level
while $J_{2n+1}$ is of an odd level for each $n\in \N$.
We define the window $W$ by
\[
 W=C\cup \bigcup_{n\in\N} J_{2n}.
\]
Observe that between two gaps of level $\ell$, there always is a gap of level $\ell+1$
so that $\partial W=C$ and $W=\overline{\textrm{int} \, W}$.

\begin{lem}\label{lem: self-similarity of W}
 Suppose we are given $x,y\in C$ with $y=\f^n(x)$ for some $n\in \Z$. 
 Then there is $\eps>0$ such that $\f^n(B_\eps(x)\cap W)=B_\eps(y)\cap W$.
\end{lem}

\begin{proof}
 Lemma~\ref{lem: self similar cantor set} yields $\eps>0$ such that $\f^n(B_{\eps}(x)\cap C)=B_{\eps}(y)\cap C$.
 In particular, each left/right endpoint $x'\in B_\eps(x)$ of a gap $J_n$ which intersects $B_\eps(x)$ corresponds to a left/right
 endpoint $y'=\f^{n}(x')\in B_\eps(y)$ of a gap $J_{n'}$ which intersects $B_\eps(y)$.
 As $J_n$ and $J_{n'}$ thus have endpoints of one and the same orbit,
 the above discussion shows that, by definition of $W$, $J_n\ssq W$ if and only if $J_{n'}\ssq W$ and hence
 $\f^n(B_\eps(x)\cap W)=B_\eps(y)\cap W$.
%
\end{proof}

\begin{rem}\label{rem: simpler self-similar window}
 Here, we see the advantage of the Cantor set $C$ over the alternative set 
 $\tilde C$ discussed in Remark~\ref{rem: simpler self-similar boundary}:
 A similar analysis as the one before shows that \emph{all} points of $\tilde C$ which are accessible
 from the left belong to \emph{one} orbit as do all points which are accessible from the right.
 Hence, by filling some but not all gaps of $\tilde C$, we have that along those orbits which correspond
 to accessible points of $\tilde C$ there are at least two local configurations of $\tilde W$
 so that $\tilde W$ is not perfectly self-similar.
 
 Of course, in order to overcome this problem, we can fill every gap partially: With the above notation, put
 $\tilde W=\tilde C\cup \bigcup_{n\in\N} [y_n,x_n]$.
 The window $\tilde W$ would be perfectly self-similar but its boundary would contain isolated points
 and thus not be a Cantor set anymore.
\end{rem}

Next, we turn to the construction of the window $V$.
Let $(J_n)$ be some
enumeration of the gaps of $C$.
Given a gap $J_n$ and some level $k\in \N_{\geq2}$, let $J(k;J_n)$
be a $k$-level gap which minimises the distance to $J_n$.
We set
\begin{align*} 
 V=\T^1\setminus \bigcup_{k\in\N_{\geq 2}} J(k;J_k).
\end{align*}
Clearly, $\partial V=C$ and $\overline{\interior V}=V$.
Moreover, if $x,\f^n(x) \in C$, then there is $\eps>0$ such that
\[\f^n(B_\eps(x)\cap V^c)\bigcap B_\eps(\f^n(x))\cap V^c=\emptyset\]
since $V$ has exactly one gap of each level and since $C$ is self-similar according to Lemma~\ref{lem: self similar cantor set}.

\begin{rem}
We would like to close this paragraph with a remark on the dependence of the topological entropy 
of a Delone dynamical system on its window.
In the following, consider the CPS  $(\R,\R,\mc L)$ of this section 
and let $(J_n)$ as well as $C$ be as above.
Given a sequence $x\in\{0,1\}^\N$, we may associate to $x$ a set $W(x)\ssq \T^1$
by setting
\[
 W(x)=C\cup \bigcup_{n\in\N,x_n=1}J_n.
\]
We denote by $\mathbb P$ the Bernoulli measure on $\{0,1\}^\N$
with equal probability $1/2$ for both symbols $0$ and $1$.
In contrast to the results of this article, we have
\begin{thm}[{\cite[Theorem~1.1]{JaegerLenzOertel2016PositiveEntropyModelSets}; see also \cite[Theorem~8]{MarkleyPaul1979}}]
For $\mathbb P$-almost every $x\in\{0,1\}^\N$, $W(x)$ is proper
and $\Omega(\oplam(W(x)))$ has positive topological entropy.
\end{thm}

Given two sequences $x,y\in\{0,1\}^\N$, denote by $z(n;x,y)\in \{0,1\}^\N$ that sequence
which coincides with $x$ on the first $n$ entries and with $y$ on all of the remaining entries.
Suppose $x$ and $y$ are elements of $\{0,1\}^\N$ such that
$W(x)=W$ while $W(y)$ is a proper set such that $\Omega(\oplam(W(y)))$ has positive topological entropy.
Observe that for each $n\in \N$ we have that $W(z(n;x,y))$ and $W(z(n;y,x))$ are proper and
$\htop(\Omega(\oplam(W(z(n;x,y))))=\htop(\Omega(\oplam(W(y)))$ as well as
$\htop(\Omega(\oplam(W(z(n;y,x))))=0$.
This immediately yields
\begin{cor}
 Suppose we are given a CPS $(\R,\R,\mc L)$.
 Consider the class of proper windows in $\R$ equipped with the Hausdorff metric.
 The map which sends each window to the topological entropy of the corresponding
 Delone dynamical system is neither upper nor lower semicontinuous.
\end{cor} 
\end{rem}

\subsection{Higher dimensional Euclidean CPS and zero entropy.}\label{sec: higher dimensional CPS and zero entropy}
 In this section, we show how for every higher dimensional Euclidean CPS
 the irregular windows constructed above yield model sets whose associated Delone dynamical system has 
 zero topological entropy in every fibre.
 
 Consider a CPS $(\R^N,\R,\cL)$.
 Analogously to the discussion at the beginning 
 of this section, the lattice $\cL$ can be represented as $\cL = A(\Z^{N+1})$, 
 where $A = (a_{ij}) \in \GL(N+1,\R)$ and each row $(a_{ij})_{j=1}^{N+1}$ has rationally
 independent entries.  Let $v_i = (a_{1i}, \ldots, a_{Ni})^T$ and put $\omega_i =
 a_{N+1,i}$. 
 Without loss of generality, we may assume $\omega_{N+1} = 1$ and $W \ssq [0,1]$ for the rest of this section.  
 
Before we come to the main results of this subsection, let us introduce some
useful concepts and notation.
 Let $\pi : \R \to \R / \Z$ denote the canonical projection and $\pi_{i}\: \R^{N+1}\to \R$ the projection onto the $i$-th coordinate. 
 We have
 \begin{align*}
  L^* &= \pi_{N+1}(\cL) = \left\{ \sum_{i=1}^N n_i\omega_i + n_{N+1} : n_i \in \Z \right\} = \pi^{-1} \left( \left\{ \sum_{i=1}^N n_i \omega_i \mod 1 : n_i \in \Z \right\} \right).
 \end{align*}
 In other words, $L^*$ is the lift of an orbit of a $\Z^N$-rotation on $\R / \Z$ with $N$ rationally
 independent rotation numbers $\omega_1,\ldots,\w_N$.
 To each rotation number, we associate
 a set $L_i^* = \pi^{-1} \left( \left\{ n \omega_i \mod 1 : n \in \Z \right\} \right)$
 and put $A_i = \begin{pmatrix} a_{ii} & a_{i,N+1} \\ \omega_i & 1 \end{pmatrix}$. 
 Then each $A_i \in \mathrm{GL}(2,\R)$ has rationally independent rows so that $\cL_i = A_i(\Z^2) \ssq \R^2$ is an irrational lattice. Note that $\pi_2(\cL_i) = L_i^*$. 
 In this way, we associate $N$ planar CPS $(\R,\R,\cL_i)$ with window $W \ssq \R$ to a given CPS $(\R^N,\R,\cL)$ with exactly the same window $W \ssq \R$.  
 We denote the corresponding Delone dynamical systems by $(\Omega(\oplam_i(W)),\varphi_i)$. 
 Observe that we have 
 $n \omega_i \mod 1 \in W$ if and only if $nv_i - \lfloor n\omega_i \rfloor v_{N+1} \in \oplam(W)$; likewise, we have $n \omega_i \mod 1 \in W$ if and only if 
 $ na_{ii} - \lfloor n\omega_i \rfloor a_{i,N+1} \in \oplam_i(W)$.
  
 Fix $t\in \R$. 
 Given $p = nv_1 + kv_{N+1}+ \sum_{i=2}^{N} m_iv_i \in \oplam(W+t)$, put $\pmb{m}_p = (m_2, \ldots, m_N) \in \Z^{N-1}$. Note that 
 $nv_1 + k v_{N+1}+\sum_{i=2}^{N} m_iv_i \in \oplam(W+t) $ is equivalent to $n\omega_1 + k \in W+t-\sum_{i=2}^N m_i\omega_i$. 
 For $\pmb{m}=(m_2, \ldots, m_N) \in \Z^{N-1}$,
 we define the {\em pseudoline} 
  $$ G_{W+t}(\pmb{m}) = \left\{ nv_1+kv_{N+1}+ \sum_{i=2}^{N} m_iv_i : n,k \in \Z, n\omega_1 + k \in W+t-\sum_{i=2}^N m_i\omega_i \right\}. $$ 
 
 Let us mention a number of immediate and important properties of pseudolines.
 First, $\{ G_{W+t}(\pmb{m}) \mid \pmb{m} \in \Z^{N-1} \}$ partitions $\oplam(W+t)$, i.e., 
 $\oplam(W+t)=\bigsqcup_{\pmb{m}\in \Z^{N-1}} G_{W+t}(\pmb{m})$. 
 Second, the restriction of $\pi_1$ to $G_{W+t}(\pmb{m})$ is injective since $a_{11}$ and $a_{1(N+1)}$ are rationally independent.
 Third, observe that for any $p\in \oplam(W+t)$ we have
  \begin{align}\label{eq: projection of pseudolines is one-dimensional model set list of properties of pseudolines}
  \pi_1(G_{W+t}(\pmb{m}_p)) = \sum_{i=2}^N m_ia_{1i} + \oplam_1\left( W+t-\sum_{i=2}^N m_i\omega_i \right) \in \Omega(\oplam_1(W+t)).
  \end{align}
 Finally, notice that there exists $C>0$ (independent of $t$) such that for each $G_{W+t}(\pmb{m})$ we have $G_{W+t}(\pmb{m}) \ssq B_C(\ell(\pmb{m}))$,
 where $\ell(\pmb{m})$ is the line $\{\lambda \cdot (1/\w_1\cdot v_1-v_{N+1})+\sum_{i=2}^Nm_iv_i\: \lambda \in \R\}\ssq \R^N$.
 Since $A\in \mathrm{GL}(N+1,\R)$, we have $(1/\w_1\cdot v_1-v_{N+1})\notin \spann\left\{ v_2, \ldots, v_N \right\}$ and therefore 
 immediately obtain the following statement.
 \begin{lem}
  \label{lemma: number of pseudolines}
  Suppose $(\R^N,\R,\cL)$ is a CPS with proper window $W \ssq \R$.
  Then there exists $\kappa > 0$ such that for each $t\in \R$ we have
   $$ \sharp \left\{ G_{W+t}(\pmb{m}) \mid \pmb{m} \in \Z^{N-1}, G_{W+t}(\pmb{m}) \cap B_M^N(0) \neq \emptyset \right\} \leq \kappa \cdot \Leb \left( B_M^{N-1}(0) \right),$$
  where $B_M^d(0) \ssq \R^d$ denotes the $d$-dimensional $M$-ball centred at $0$. 
 \end{lem}
 
The next result provides a whole class of higher dimensional CPS with irregular windows
whose associated Delone dynamical systems have zero entropy.

 \begin{thm}
  \label{thm: zero entropy high dim}
  Let $(\R^N,\R,\cL)$ be a CPS with proper window $W \ssq \R$. 
  Furthermore, assume that there exists $i \in \{1, \ldots, N\}$ such that $\htop(\varphi_i) = 0$. 
  Then we have $\htop^\xi(\varphi) = 0$ for all $\xi \in \T$. 
 \end{thm} 

 \begin{proof}
  Without loss of generality, we may assume that $\htop(\varphi_1) = 0$.  
  We equip $\R$ as well as $\R^N$ with the Euclidean metric and consider the 
  entropy of $\varphi_1$ and $\varphi$ obtained by averaging over the van Hove sequence given by one-dimensional balls 
  ${(B_M^1(0))}_{M\in \N}$ and $N$-dimensional balls ${(B_M(0))}_{M\in \N}$, respectively.
  Fix some $\xi \in \T^{N+1}$ and assume without loss of generality that there is $t\in\R$ with $\xi= [0,t]_\cL$.
  Let $\eps>0$ be smaller than $r\coloneqq \frac{1}{2} \cdot \min\{\inf_{p\neq q\in \oplam(W+t)} \|p-q\|, \inf_{p\neq q\in \oplam_1(W)} |p-q|\}$.
  Given $M\in \N$, let $S_1(\varepsilon,M)$ be $(\varepsilon,M)$-spanning for $\Omega(\oplam_1(W))$ with minimal cardinality 
  $ P_1(\varepsilon,M)\coloneqq \sharp S_1(\varepsilon,M)$.
  Our goal is to construct a set $S^\xi(\eps,M)$ which is $(\varepsilon,M)$-spanning for $\beta^{-1}(\xi)$ 
 and satisfies
  \begin{align}\label{eq: cardinality spanning set}
\sharp S^\xi(\varphi,\varepsilon,M) \leq P_1(\varepsilon,M)^{\kappa \Leb \left( B^{N-1}_{M+1/\varepsilon}(0) \right) }.   
  \end{align}

  To that end, recall that two Delone sets $\Lambda, \Gamma \in \beta^{-1}(\xi)$ satisfy $\max_{s \in B_M(0)} d(\Lambda-s,\Gamma-s) < \varepsilon$ if 
  for all $s \in B_M(0)$ we have 
  $\Lambda \cap B_{1/\varepsilon}(s) = \Gamma \cap B_{1/\varepsilon}(s)$ (see Remark~\ref{remark: inclusion in fibres}).
  Since $\oplam(W+t)$ can be covered by pseudolines, this is the case if for all such $s$ and 
  each $p \in B_{1/\varepsilon}(s) \cap \oplam(W+t)$ we have
  $G_{W+t}(\pmb{m}_p) \cap \Lambda \cap B_{1/\varepsilon}(s)= G_{W+t}(\pmb{m}_p) \cap \Gamma \cap B_{1/\varepsilon}(s)$.
  Note that this is equivalent to
  \begin{align}\label{eq: close model sets coincide on pseudolines}
   \pi_1(G_{W+t}(\pmb{m}_p) \cap \Lambda \cap B_{1/\varepsilon}(s)) = \pi_1(G_{W+t}(\pmb{m}_p) \cap \Gamma \cap B_{1/\varepsilon}(s)),
   \end{align}
   since $\left.\pi_1\right|_{G_{W+t}(\pmb{m}_p)}$ is injective.
   
   Now, given $\Gamma\in \beta^{-1}(\xi)$ with $\Gamma =\lim_{j\to\infty} \oplam(W+t_j)$ (see \eqref{eq:equivalenceForFibre}),
   observe that
  \begin{align}\label{eq: projection of pseudolines is one-dimensional model set}
  \begin{split}
   \pi_1(G_{W+t}(\pmb{m})\cap \Gamma)& = \pi_1(G_{W+t}(\pmb{m})\cap \lim_{j\to\infty} \oplam(W+t_j))\\
  &=\lim_{j\to\infty} \pi_1(G_{W+t_j}(\pmb{m})\cap  \oplam(W+t_j))
  \end{split}
  \end{align}
  which is an element of $\Omega(\oplam_1(W))$ due to \eqref{eq: projection of pseudolines is one-dimensional model set list of properties of pseudolines}.
  Hence, by definition of $S_1(\eps,M)$, there is $\Delta \in S_1(\varepsilon,M)$ with 
    $\max_{s \in B^1_M(0)} d(\pi_1\left(G_{W+t}(\pmb{m}) \cap \Gamma\right)-s,\Delta-s) < \varepsilon$.
    In particular, we have
    \begin{align}\label{eq: defn eps-configuration}
 \pi_1\left(G_{W+t}(\pmb{m}) \cap \Gamma \cap B_{M+1/\varepsilon}(0)\right) \ssq \Delta+\delta  
  \end{align}
  for some $\delta\in \R$ with 
  $|\delta|<\eps$.
  Since $\eps<r$, we have that for fixed $\pmb{m}$ and $\Delta$ there is at most one such $\delta$ 
  for which \eqref{eq: defn eps-configuration} is satisfied for some $\Gamma \in \oplam(W+t)$.
  If \eqref{eq: defn eps-configuration} holds, we say \emph{$\Gamma$ realises the local configuration of $\Delta$ along $G_{W+t}(\pmb{m})$}.
  We define an equivalence relation $\sim$ on $\beta^{-1}(\xi)$ by putting $\Gamma \sim \Lambda$ if
  $\Gamma $ and $\Lambda$ realise the same local configuration along $G_{W+t}(\pmb{m}_p)$
  ($p\in B_{1/\varepsilon}(s) \cap \oplam(W+t)$).   
  The above shows: $\max_{s \in B_M(0)} d(\Lambda-s,\Gamma-s) < \varepsilon$ if $\Lambda \sim \Gamma$.
   
  Finally, we set $S^\xi(\varphi,\varepsilon,M)$ to be a set which contains one representative for each equivalence class
  of $\sim$.
  Recall that the number of pseudolines that intersect $B_{M+1/\varepsilon}(0)$ is bounded by $\kappa\cdot \Leb ( B^{N-1}_{M+1/\varepsilon}(0))$ (see Lemma~\ref{lemma: number of pseudolines}). 
  Since there are at most $ P_1(\varepsilon,M)$ possible configurations realised along each 
  $G_{W+t}(\pmb{m}_p) \cap B_{M+1/\varepsilon}(0)$,
  we obtain \eqref{eq: cardinality spanning set}.
  Thus,
  \begin{align*}
   \htop^\xi(\varphi) &\leq \lim_{\varepsilon \to 0} \limsup_{M \to \infty} \frac{\kappa \Leb \left( B^{N-1}_{M+1/\varepsilon}(0) \right)}{\Leb \left( B_M(0) \right)} \log P_1(\varepsilon,M)  \\
                          &= \lim_{\varepsilon \to 0} \limsup_{M \to \infty} \frac{2\kappa}{\sqrt{\pi}} \left( \frac{1}{M\varepsilon}+1 \right)^{N-1} \frac{\log P_1(\varepsilon,M)}{\Leb(B^1_M(0))} \\
                          &\leq \frac{2\kappa}{\sqrt{\pi}} \lim_{\varepsilon \to 0} \limsup_{M \to \infty} \frac{\log P_1(\varepsilon,M) }{\Leb(B^1_M(0))}= 0
  \end{align*}
  which finishes the proof.
  \end{proof}

\section{Invariant measures, dynamical spectrum and diffraction} \label{Spectra}
%
In this part, we discuss the spectral properties of the model sets constructed in the previous sections.
Suppose $G$ is an abelian group.
 Recall that given a topological dynamical system $(X,G)$
 which preserves a measure $\mu$, we say $f\in L_2(X,\mu)$
 is an \emph{eigenfunction} of $(X,G)$ (equipped with $\mu$) if there exists $\lambda\in \hat G$ such that
 $g.f=\lambda(g)\cdot f$ ($g\in G$), where
 $g.f(x)=f(gx)$ ($x\in X$). 
 Here, $\hat G$ denotes the dual of $G$.
 We say $(X,G)$ has \emph{pure point spectrum} if there exists an orthonormal basis of $L_2(X,\mu)$ which consists of eigenfunctions.

Let us recall some basic facts from the spectral theory of minimal equicontinuous topological dynamical systems
$(\T,G)$.
Due to Theorem~\ref{thm: representation of general equicont sys}, we may assume without loss of generality that $\T$ is a compact abelian group
and $g\xi =\xi+\w(g)$ for all $\xi \in \T$ and $g\in G$, where $\w\: G\to \T$ denotes a 
group homomorphism with dense image in $\T$.
Note that for all $\lambda\in \hat \T$, $g\in G$, and $\xi\in \T$ we have
$\lambda(\xi+\w(g))=\lambda(\w(g))\cdot \lambda(\xi)$.
Moreover, observe that $\lambda(\w(\cdot))$ is a character on $G$.
Hence, every element of $\hat \T$ is an eigenfunction of $(\T,G)$ (equipped with the unique invariant measure $\Theta_\T$).
Recall that by the Peter Weyl Theorem, the characters of a compact group $\T$ form an orthonormal basis
of $L_2(\T)=L_2(\T,\Theta_\T)$.
This shows the following well-known fact: every minimal equicontinuous system has pure point spectrum
with continuous eigenfunctions.


In the following, we consider a Delone dynamical system $(\Omega(\oplam(W)),G)$
corresponding to a CPS $(G,H,\mc L)$ with proper window $W$.
As before, let $\beta\:\Omega(\oplam(W))\to \T$ be the associated torus parametrisation,
$\varphi$ the translation action on $(\Omega(\oplam(W))$ (see Section~\ref{sec: delone dynamical systems})
and $\w$ the $G$-action on $\T$ (see Section~\ref{subsec:torusParametrisation}).
 \begin{definition}
 A measurable map $\gamma\:\T\to \Omega(\oplam(W))$ is referred to as an \emph{invariant graph} (for $(\Omega(\oplam(W))$)
 if
 \begin{align}\label{eq: invariant graph}
\forall  s,u\in G,\, t\in H\: \quad \beta\circ\gamma([s,t]_{\mc L})=[s,t]_{\mc L} \quad \text{ and }\quad   \varphi(u,\gamma([s,t]_{\mc L}))=\gamma(\w(u,[s,t]_{\mc L})).
 \end{align}
 \end{definition}
 
 Given an invariant graph $\gamma$, we define the associated \emph{graph measure} by setting
\[
 \mu_\gamma(A)=\Theta_\T(\gamma^{-1}(A))
\]
for all measurable $A\ssq \Omega(\oplam(W))$.
Observe that, since $\Theta_\T$ is ergodic, $\mu_\gamma$ is an ergodic measure of $(\Omega(\oplam(W)),G)$.
Define
\[
 U_\gamma\: L_2(\Omega(\oplam(W)),\mu_\gamma)\to L_2(\T),\qquad f\mapsto f\circ \gamma.
\]
Observe that
\begin{align*}
 \langle U_\gamma f, U_\gamma g\rangle_{L_2(\T)}&=\int_\T \! \overline{U_\gamma f} \cdot U_\gamma g\, d\Theta_\T=
 \int_\T \! (\overline{f} \cdot g) \circ \gamma\,d\Theta_\T =\int_{\Omega(\oplam(W))}\! \overline f\cdot g\, d\mu_{\gamma}\\
 &=\langle f, g\rangle_{L_2(\Omega(\oplam(W)),\mu_\gamma)},
\end{align*}
for all $f,g\in L_2(\Omega(\oplam(W)),\mu_\gamma)$.
Furthermore, due to \eqref{eq: invariant graph}, we have $U_\gamma (g\circ \beta)=g$ for all $g\in L_2(\T)$.
Hence, $U_\gamma$ is bijective.
Finally, for each $u\in G$, \eqref{eq: invariant graph} yields
\[
 u. (U_\gamma f)\, (\cdot)=f\circ \gamma (\w(u,\cdot))=f(\varphi(u,\gamma(\cdot))))=(u.f)\circ \gamma\, (\cdot)=
 U_\gamma (u.f)\, (\cdot).
\]
Altogether, we have proven
\begin{prop}
$(\Omega(\oplam(W)),G)$ equipped with a graph measure has pure point spectrum and all eigenfunctions are continuous.
\end{prop}

Suppose for almost every $[s,t]_{\mc L}\in \T$, the fibre $\beta^{-1}([s,t]_{\mc L})$ contains
a unique maximal element $\Gamma_+$ or a unique minimal element $\Gamma_-$ with respect to set inclusion.
Given the existence of such elements, we set
\[
 \gamma_{\pm}\:\T\to\Omega(\oplam(W)),\qquad [s,t]_{\mc L}\mapsto \Gamma_\pm([s,t]_{\mc L}).
\]

\begin{prop}
 Suppose almost every fibre contains an element $\Gamma_+$ ($\Gamma_-$)
 as above.
 Then $\gamma_+$ ($\gamma_-$) is an invariant graph.
\end{prop}
\begin{proof}
 As the proofs for $\gamma_+$ and $\gamma_-$ are similar, we omit the index $\pm$ in the following.
 \eqref{eq: invariant graph} is obvious.
 In order to see the measurability of $\gamma$, recall that 
 $F\:\T\to\mc K(\Omega(\oplam(W))),\, [s,t]_{\mc L}\mapsto \beta^{-1}([s,t]_{\mc L})$, where $\mc K(\Omega(\oplam(W)))$
 denotes the class of compact subsets of $\Omega(\oplam(W))$, is measurable (see the proof of Theorem~\ref{t.tame_implies_regular_for_group_actions}). 
 Lusin's Theorem hence yields the existence of compact sets $K_n\ssq \T$ with $\Theta_\T(K_n)>1-1/n$ on 
 which $F$ is continuous with respect to the Hausdorff topology on $\mc K(\Omega(\oplam(W)))$.
 Observe that hence, $\left.\gamma\right|_{K_n}$ is continuous, too.
 Thus, $\gamma$ is measurable with respect to the completion of the sigma algebra of the Borel sets of 
 $\T$.
\end{proof}

\begin{rem}
 In the situation of the previous statement, the measure $\mu_{\gamma_+}$ is also referred to as
 \emph{Mirsky measure} (see, for example, \cite{Kulaga-PrzymusLemanczykWeiss2015,KellerRichard2018b,kellerRichard2018a}).
\end{rem}

Now, let $W$ and $V$ be as in Section~\ref{MainConstruction}.
By means of Remark~\ref{rem: two-to-one extension with upper and lower graph} and the above statement, we see that
$(\Omega(\oplam(W)),\R)$ allows for two invariant graphs $\gamma_\pm$ (mapping each $\xi\in \T$ to the maximal [minimal] element of $\beta^{-1}(\xi)$).
Moreover, taking into account that 
the torus parametrisation of $(\Omega(\oplam(W)),\R)$
is almost everywhere $2$-to-$1$, 
the proof of Lemma~\ref{lem: locally disjoint complements implies unique ergodicity} also shows that the associated graph measures
$\mu_{\gamma_\pm}$ of $(\Omega(\oplam(W)),\R)$ are the only ergodic measures
of $(\Omega(\oplam(W)),\R)$.
Likewise, due to Lemma~\ref{lem:characterisationOfLDCWindows} and the unique ergodicity of $(\Omega(\oplam(V)),\R)$,
we have that $(\Omega(\oplam(V)),\R)$ allows for a unique invariant 
graph $\gamma$ (mapping each $\xi\in \T$ to the maximal element of $\beta^{-1}(\xi)$).
We have thus proven
\begin{thm}
 $(\Omega(\oplam(W)),\R)$ (equipped with any ergodic measure) as well as
 $(\Omega(\oplam(V)),\R)$ (equipped with the unique invariant measure)
 have pure point dynamical spectrum with all eigenfunctions being continuous.
\end{thm}

By an immediate application of \cite[Theorem~3.2]{LeeMoodySolomyak2002} and
\cite[Theorem~4.1]{LeeMoodySolomyak2002} we hence get

\begin{cor}
 Every $\Gamma\in (\Omega(\oplam(V)),\R)$ and $\mu$-almost every $\Gamma \in (\Omega(\oplam(W)),\R)$ (where $\mu$ is any invariant measure on $(\Omega(\oplam(W)),\R)$)
 has pure point diffraction spectrum.
\end{cor}


\begin{thebibliography}{10}

\bibitem{Meyer1972AlgebraicNumbers}
Y.~Meyer.
\newblock {\em Algebraic numbers and harmonic analysis}.
\newblock North-Holland Publishing Co., Amsterdam-London; American Elsevier
  Publishing Co., Inc., New York, 1972.
\newblock North-Holland Mathematical Library, Vol. 2.

\bibitem{Schlottmann1999GeneralizedModelSets}
M.~Schlottmann.
\newblock Generalized model sets and dynamical systems.
\newblock In {\em Directions in mathematical quasicrystals}, volume~13 of {\em
  CRM Monogr. Ser.}, pages 143--159. Amer. Math. Soc., Providence, RI, 2000.

\bibitem{LeeMoodySolomyak2002}
J.-Y. Lee, R.~V. Moody, and B.~Solomyak.
\newblock Pure point dynamical and diffraction spectra.
\newblock {\em Ann. Henri Poincar\'e}, 3(5):1003--1018, 2002.

\bibitem{BaakeLenz2004PurePointDiffractionSpectrum}
M.~Baake and D.~Lenz.
\newblock Dynamical systems on translation bounded measures: Pure point
  dynamical and diffraction spectra.
\newblock {\em Ergodic Theory Dynam.\ Systems}, 24(6):1867--1893, 2004.

\bibitem{BaakeLenzMoody2007Characterization}
M.~Baake, D.~Lenz, and R.V. Moody.
\newblock Characterization of model sets by dynamical systems.
\newblock {\em Ergodic Theory Dynam. Systems}, 27(2):341--382, 2007.

\bibitem{Glasner2018MinimalTameSystems}
E.~Glasner.
\newblock The structure of tame minimal dynamical systems for general groups.
\newblock {\em Invent. Math.}, 211(1):213--244, 2018.

\bibitem{PleasantsHuck2013KFreePoints}
P.A.B. Pleasants and C.~Huck.
\newblock Entropy and diffraction of the $k$-free points in $n$-dimensional
  lattices.
\newblock {\em Discrete \& Computational Geometry}, 50(1):39--68, 2013.

\bibitem{BaakeJaegerLenz2016ToeplitzFlowsModelSets}
M.~Baake, T.~J{\"a}ger, and D.~Lenz.
\newblock Toeplitz flows and model sets.
\newblock {\em Bull.\ Lond.\ Math.\ Soc.}, 48(4):691--698, 2016.

\bibitem{Williams1984ToeplitzFlows}
S.~Williams.
\newblock Toeplitz minimal flows which are not uniquely ergodic.
\newblock {\em Z. Wahrsch. verw. Gebiete}, 67(1):95--107, 1984.

\bibitem{MarkleyPaul1979}
N.G. Markley and M.E. Paul.
\newblock Almost automorphic symbolic minimal sets without unique ergodicity.
\newblock {\em Israel J. Math.}, 34(3):259--272 (1980), 1979.

\bibitem{Fomin1951}
S.~Fomin.
\newblock On dynamical systems with a purely point spectrum.
\newblock {\em Doklady Akad. Nauk SSSR (N.S.)}, 77:29--32, 1951.

\bibitem{Auslander1959}
J.~Auslander.
\newblock Mean-{$L$}-stable systems.
\newblock {\em Illinois J. Math.}, 3:566--579, 1959.

\bibitem{LiTuYe2015}
J.~Li, S.~Tu, and X.~Ye.
\newblock Mean equicontinuity and mean sensitivity.
\newblock {\em Ergodic Theory Dynam. Systems}, 35(8):2587--2612, 2015.

\bibitem{DownarowiczGlasner2015IsomorphicExtensionsAndMeanEquicontinuity}
T.\ Downarowicz and E.~Glasner.
\newblock Isomorphic extensions and applications.
\newblock {\em Topol. Methods Nonlinear Anal.}, 48(1):321--338, 2016.

\bibitem{FuhrmannGrogerLenz2018}
G.~Fuhrmann, M.~Gr{\"o}ger, and D.~Lenz.
\newblock The structure of mean equicontinuous group actions.
\newblock In preparation.

\bibitem{auslander1988minimal}
J.~Auslander.
\newblock {\em Minimal flows and their extensions}.
\newblock Elsevier Science Ltd, 1988.

\bibitem{HewittRoss1963}
E.~Hewitt and K.A. Ross.
\newblock {\em Abstract harmonic analysis. {V}ol. {I}: {S}tructure of
  topological groups. {I}ntegration theory, group representations}.
\newblock Die Grundlehren der mathematischen Wissenschaften, Bd. 115. Academic
  Press, Inc., Publishers, New York; Springer-Verlag,
  Berlin-G\"{o}ttingen-Heidelberg, 1963.

\bibitem{Koehler1995EnvelopingSemigroups}
  A.~K{\"ohler}.
  \newblock Enveloping semigroups for flows.
  \newblock {\em Proceedings of the Royal Irish Academy}, 95A:179--191, 1995.

\bibitem{Glasner2006OnTameSystems}
  E.~Glasner.
  \newblock On tame dynamical systems.
  \newblock {\em Colloq. Math.}, 105:283--295, 2006. 

  \bibitem{Huang2006TameSystemsAndScrambledPairs}
    W.~Huang.
    \newblock Tame systems and scrambled pairs under an abelian group action.
    \newblock {\em Ergodic Theorie Dyn. Sys.}, 26:1549--1567, 2006.
    
\bibitem{Glasner2006TameSystems}
E.~Glasner.
\newblock The structure of tame minimal dynamical systems.
\newblock {\em Ergodic Theory Dyn. Syst.}, 27(6):1819--1837, 2007.

\bibitem{KerrLi2007Independence}
D.~Kerr and H.~Li.
\newblock Independence in topological and {C*}-dynamics.
\newblock {\em Math.\ Ann.}, 338(4):869--926, 2007.

\bibitem{Lindenstrauss2001}
E.~Lindenstrauss.
\newblock Pointwise theorems for amenable groups.
\newblock {\em Invent. Math.}, 146(2):259--295, 2001.

\bibitem{bowen:EntropyForGroupEndomorphismsAndHomogeneousSpaces}
R.~Bowen.
\newblock {Entropy for group endomorphisms and homogeneous spaces}.
\newblock {\em {Trans. Am. Math. Soc.}}, pages 401--413, 1971.

\bibitem{baake2007purePointDiffraction}
M.~Baake, D.~Lenz, and C.~Richard.
\newblock Pure point diffraction implies zero entropy for delone sets with
  uniform cluster frequencies.
\newblock {\em Letters in Mathematical Physics}, 82(1):61--77, 2007.

\bibitem{hauser}
T.~Hauser.
\newblock Some remarks on entropy of topological group actions.
\newblock In preparation.

\bibitem{MullerRichard2013}
P.~M\"{u}ller and C.~Richard.
\newblock Ergodic properties of randomly coloured point sets.
\newblock {\em Canad. J. Math.}, 65(2):349--402, 2013.

\bibitem{Moody2000ModelSetsSurvey}
R.V. Moody.
\newblock Model sets: A survey.
\newblock In {\em From quasicrystals to more complex systems}, pages 145--166.
  Springer, 2000.

\bibitem{BaakeGrimm2013AperiodicOrder}
M.\ Baake and U.~Grimm.
\newblock {\em {Aperiodic Order}}, volume~1.
\newblock Cambridge Univ.\ Press, 2013.

\bibitem{lee2006characterization}
Jeong-Yup Lee and Robert~V Moody.
\newblock A characterization of model multi-colour sets.
\newblock In {\em Annales Henri Poincar{\'e}}, volume~7, pages 125--143.
  Springer, 2006.

\bibitem{Paul1976}
M.E. Paul.
\newblock Construction of almost automorphic symbolic minimal flows.
\newblock {\em General Topology and Appl.}, 6(1):45--56, 1976.

\bibitem{DownarowiczDurand2002}
T.~Downarowicz and F.~Durand.
\newblock Factors of {T}oeplitz flows and other almost {$1-1$} extensions over
  group rotations.
\newblock {\em Math. Scand.}, 90(1):57--72, 2002.

\bibitem{Downarowicz}
T.~Downarowicz.
\newblock Private communication.

\bibitem{JaegerLenzOertel2016PositiveEntropyModelSets}
T.\ J{\"a}ger, D.\ Lenz, and C.~Oertel.
\newblock Model sets with positive entropy in {E}uclidean cut and project
  schemes.
\newblock Preprint {\tt arXiv:1605.01167}, to appear in {\em Ann.\ Sci. {\'Ecole}\
  Norm.\ Sup.}, 2016.

\bibitem{DownarowiczKasjan2015}
T.~Downarowicz and S.~Kasjan.
\newblock Odometers and {T}oeplitz systems revisited in the context of
  {S}arnak's conjecture.
\newblock {\em Studia Math.}, 229(1):45--72, 2015.

\bibitem{Swierczkowski1959}
S.~\'Swierczkowski.
\newblock On successive settings of an arc on the circumference of a circle.
\newblock {\em Fund. Math.}, 46:187--189, 1959.

\bibitem{deMelovanStrien1993}
Welington de~Melo and Sebastian van Strien.
\newblock {\em One-dimensional dynamics}, volume~25 of {\em Ergebnisse der
  Mathematik und ihrer Grenzgebiete (3) [Results in Mathematics and Related
  Areas (3)]}.
\newblock Springer-Verlag, Berlin, 1993.

\bibitem{Kurka2003}
Petr K\r{u}rka.
\newblock {\em Topological and symbolic dynamics}, volume~11 of {\em Cours
  Sp\'ecialis\'es [Specialized Courses]}.
\newblock Soci\'et\'e Math\'ematique de France, Paris, 2003.

\bibitem{Kulaga-PrzymusLemanczykWeiss2015}
J.~Ku{\l}aga-Przymus, M.~Lema\'{n}czyk, and B.~Weiss.
\newblock On invariant measures for {$\mathcal{B}$}-free systems.
\newblock {\em Proc. Lond. Math. Soc. (3)}, 110(6):1435--1474, 2015.

\bibitem{KellerRichard2018b}
G.~Keller and C.~Richard.
\newblock Dynamics on the graph of the torus parametrization.
\newblock {\em Ergodic Theory Dynam. Systems}, 38(3):1048--1085, 2018.

\bibitem{kellerRichard2018a}
G.~Keller and C.~Richard.
\newblock Periods and factors of weak model sets.
\newblock {\em Israel J. Math.}, pages 1--48, 9 2018.

\end{thebibliography}
\end{document}